\documentclass[a4paper,12pt]{article}
\usepackage{amsfonts}
\usepackage{amsmath}
\usepackage{amssymb}
\usepackage{amsthm}
\usepackage{graphicx}
\usepackage{tikz}
\usepackage[hidelinks]{hyperref}
\usepackage{marginnote}
\usepackage[labelformat=simple]{subcaption}

\usepackage[inline]{enumitem}
\usepackage{pdflscape}
\usepackage{afterpage}
\usepackage{color, colortbl}

\usepackage{array}
\newcolumntype{H}{>{\setbox0=\hbox\bgroup}c<{\egroup}@{}}

\usetikzlibrary{arrows,shapes, positioning, matrix, patterns}

\newtheorem{prop}{Proposition}
\numberwithin{prop}{section}
\newtheorem{lem}{Lemma}
\numberwithin{lem}{section}
\newtheorem{coro}{Corollary}
\numberwithin{coro}{section}

\numberwithin{teo}{section}

\theoremstyle{definition}
\newtheorem{defix}{Definition}
\numberwithin{defix}{section}
\newenvironment{defi}
  {\pushQED{\qed}\defix}
  {\popQED\enddefix}
\newtheorem{exex}{Example}
\numberwithin{exex}{section}
\newenvironment{exe}
  {\pushQED{\qed}\exex}
  {\popQED\endexex}

\newtheorem{obsx}{Remark}
\numberwithin{obsx}{section}
\newenvironment{obs}
  {\pushQED{\qed}\obsx}
  {\popQED\endobsx}

\newcommand{\sign}{\textnormal{sign}}

\providecommand{\keywords}[1]{\textit{Keywords:  } #1}

\providecommand{\MSC}[1]{\textit{2010 Mathematics subject classification:  } #1}

\title{The Lifting Bifurcation Problem \\on Feed-Forward Networks}
\author{
Pedro Soares\thanks{Departamento de  Matem\'{a}tica,
Faculdade de Ci\^{e}ncias  da Universidade do Porto,
Centro de Matem\'{a}tica da Universidade do Porto
    (ptcsoares@fc.up.pt).}
}
\date{}

\begin{document}
\maketitle

\begin{abstract}
\noindent
We consider feed-forward networks, that is, networks where cells can be divided into layers, such that every edge targeting a layer, excluding the first one, starts in the prior layer.
A feed-forward system is a dynamical system that respects the structure of a feed-forward network.
The synchrony subspaces for a network, are the subspaces defined by equalities of some cells coordinates, that are flow-invariant by all the network systems.
The restriction of each network system to each synchrony subspace is a system associated with a smaller network, which may be, or not, a feed-forward network.
The original network is then said to be a lift of the smaller network.
We show that a feed-forward lift of a feed-forward network is given by the composition of two types of lifts: lifts that create new layers and lifts inside a layer.
Furthermore, we address the lifting bifurcation problem on feed-forward systems. 
More precisely, the comparison of the possible codimension-one local steady-state bifurcations of a feed-forward system and those of the corresponding lifts is considered.
We show that for most of the feed-forward lifts, the increase of the center subspace is a sufficient condition for the existence of additional bifurcating branches of solutions, which are not lifted from the restricted system.
However, when the bifurcation condition is associated  with the internal dynamics and the lifts occurs inside an intermediate layer, we prove that the existence of bifurcating branches of solutions that are not lifted from the restricted system does depend generically on the particular feed-forward system.
%
%Furthermore, we study the bifurcation branch on feed-forward systems and whenever the bifurcation branches occurring in a feed-forward system associated to a lift network are lifted from the original network. 
%A necessary condition to the existence of not lifted bifurcation branches associated to a lift network is that the center subspace of the feed-forward systems associated to the lifted network and the original network have different dimensions.
%Often this condition is also sufficient.
%Some feed-forward systems with a bifurcation condition associated to the internal dynamics and lifts inside a intermediate layer are an exception. 
%In those cases, the bifurcation branches are not lifted for an open set of feed-forward system and are lifted for another open set of feed-forward system.

\vspace{1em}
\hspace{-1.8em}

\vspace{1em}
%\noindent
\hspace{-1.8em}
\MSC{37G10; 34D06}
\end{abstract}

\section{Introduction}

%COUPLED CELL NETWORKS

Coupled cell networks describe influences between cells and are represented by graphs. 
A coupled cell system is characterized by an admissible vector field that respects the structure of the network and that governs the dynamics of each cell. 
Roughly speaking, a coupled cell system respects the structure of the network if the dynamics of a cell affects the dynamics of other cell, whenever there exists an edge on the network from the former to the latter. 
In \cite{SGP03, GST05}, the authors formalized the concepts of coupled cell network and coupled cell system. 
They also showed that there exists an intrinsic relation between coupled cell systems and networks, proving that a subspace defined by equalities of some of the network cell coordinates is an invariant subspace for any coupled cell system if and only if the coloring on the network set of cells determined by those equalities is balanced. 
Given a balanced coloring, the associated quotient network is given by merging cells with equal color. 
The original network is then said to be a lift of the smaller network.
%Moreover, the restriction of a coupled cell system to an invariant synchrony subspace is a coupled cell system for the quotient network associated to the balanced coloring defining the synchrony subspace.
%The adjacency matrix of a quotient network is given by the restriction of the original adjacency matrix to the correspondent synchrony subspace, \cite{ADGL09}.
%We will focus on regular network, all cell and edges are identical, which can be represented by its adjacency matrix. 

%Feed-forward networks
Feed-forward networks correspond to a class of coupled cell networks are a particular coupled cell networks where the cells can be partitioned into layers. 
Consider feed-forward networks where each cell in the first layer only receives inputs from itself, cells in the other layers only receive inputs from cells in the previous layer and any two cells receive the same number of inputs.
Feed-forward networks have been studied in \cite{EG06,BGV07,GP12,RS13,G14,ADF17}. 
In \cite{ADF17}, the authors studied the balanced colorings of feed-forward networks. 
We will focus on lifts of feed-forward networks that have also a feed-forward structure.
Those lifts are given by the composition of two basic types of lifts: lifts that create new layers and lifts inside a layer. See Proposition~\ref{prop:compolift}. 
A lift that creates new layers is the replication of the first layer into consecutive layers.
A lift inside a layer is given by the division of cell within some layer.

% Bifrucation on feed-forward networks
%The study of bifurcations on feed-forward networks revealed surprising phenomenons, see e.g.  \cite{EG06, BGV07, GP12,RS13}.
In \cite{EG06, GP12}, the authors studied Hopf bifurcation on feed-forward networks and proved that the appearing periodic orbits have an amplitude of surprising order.
They assumed the phase space of each cell to be the two dimensional real space. 
In this work, we restrict our attention to one dimensional cell phase space and study (codimension-one steady-state) bifurcations at a full synchrony equilibrium.
It follows from the feed-forward structure that the Jacobian matrix of a feed-forward system at a full synchrony equilibrium has only two eigenvalues: the valency and the internal dynamics.
In this paper, we study the two different kinds of bifurcations associated to each eigenvalue.
In \cite{RS13}, the authors also studied bifurcations associated with the internal dynamics on feed-forward networks which have a semi-group structure and only one cell in each layer. 
They introduced, what we call, square-root-order of a bifurcation branch that measure the growth of a bifurcation branch, Definition~\ref{def:RS1322}.
Exploiting the techniques presented in \cite{RS13}, we give a characterization of the steady-state bifurcation branches associated with the internal dynamics on a feed-forward network in terms of their square-root-orders and slopes. See Proposition~\ref{prop:equivtheta}. 

%Lifting bifurcation problem

Last, we study the lifting bifurcation problem on feed-forward networks for the two basic types of lifts on feed-forward networks.
The restriction of a lift system to a synchrony subspace is a feed-forward system. 
Thus any bifurcation branch occurring for a feed-forward system corresponds to a bifurcation branch occurring for the lift system.
The lifting bifurcation problem asks if there are more bifurcation branches occurring for the lift system.
This problem was first raised in \cite{ADGL09} where the authors proved that there are networks which have more bifurcations branches on some lift systems than the ones lifted from the original network. 
This problem was also studied in \cite{M14, NRS17}.
A well-know result gives a necessary condition for the lifting bifurcation problem:  
There can be more bifurcation branches on a lifted network than the ones lifted only if the center subspace of the coupled cell systems associated to the original network and the lift network have different dimensions. See Corollary~\ref{coro:eigspainv}.

Frequently, this condition is also sufficient for the lifting bifurcation problem and we prove it in the following cases.
For lifts that create new layers and inside a layer, and feed-forward systems with a bifurcation condition associated to the valency, Proposition~\ref{prop:lbpval}.
For lifts that create new layers, inside the first layer and inside a layer which has only one cell in the next layer, and feed-forward systems with a bifurcation condition associated to the internal dynamics, Propositions~\ref{prop:lbpintdynfirstnew} and \ref{prop:lbpintdynextonecell}.
Moreover, the previous cases do not depend generically on the feed-forward system.

Considering lifts inside an intermediate layer and feed-forward systems with a bifurcation condition associated to the internal dynamics, the center subspaces of the feed-forward system and of the lift system have different dimensions.
We show, for lifts inside an intermediate layer, that there exists an open set of coupled cell systems with a bifurcation associated to the internal dynamics such that there are more bifurcation branches on the lifted network than the ones lifted from the original network, Proposition~\ref{prop:liftbifbrainliftinsidelayerusingbalcol}. 
Remarkably, for a class of feed-forward networks and lifts inside an intermediate layer, there is also an open set of feed-forward systems with a bifurcation associated to the internal dynamics such that there are no more bifurcation branches on the feed-forward lift, Propositions~\ref{prop:genkerincnonewbifseclay} and \ref{prop:genkerincnonewbif}.

%WHAT WE DO
%SUMMARY

The paper is organized as follows. 
In Section~\ref{sec:ffn}, we recall the definitions of coupled cell networks and feed-forward networks.
Next, we study lifts of a feed-forward network which preserve the feed-forward structure (Section~\ref{sec:liftffn}).
Coupled cell systems and feed-forward systems are recalled in Section~\ref{sec:ffs}.
Then we analyze the steady-state bifurcations in feed-forward systems with bifurcation conditions associated to the valency (Section~\ref{sec:bifffnval}) and the internal dynamics (Section~\ref{sec:ffnssbar}).
Finally, we study the lifting bifurcation problem for feed-forward systems with bifurcation conditions associated to the valency (Section~\ref{sec:lbpffnval}) and the internal dynamics (Section~\ref{sec:lbfffnintdy}).
Many results assume conditions on the feed-forward networks and we provide, throughout the paper, examples showing that those conditions are necessary.

\section{Feed-forward networks}\label{sec:ffn}

In this section, we recall a few facts concerning coupled cell networks, following \cite{GST05, RS14}, and define feed-forward networks.

\begin{defi}
A \emph{network} $N$ is defined by a directed graph with a finite set of cells $C$ and a finite sets of directed edges divided by types $E_1,\dots,E_k$. 
We assume that each cell $c$ is target by one and only one edge of each type.
We denote by $|N|$ the number of cells in the network $N$. % and we say that the network $N$ has $k$ input-types.

Let $(\sigma_{i}: C \rightarrow C)_{i=1}^{k}$ be the collection of functions such that there exists an edge $e\in E_i$ from $\sigma_{i}(c)$ to $c$, for every $c\in C$ and $1\leq i\leq k$. We say that $N$ is \emph{represented by the functions $(\sigma_{i}: C \rightarrow C)_{i=1}^{k}$}.
\end{defi}

Let $N$ and $N'$ be two networks represented by the functions  $\{\sigma_i:C\rightarrow C\}_{i=1}^k$ and $\{\sigma'_i:C'\rightarrow C'\}_{i=1}^k$, respectively.
We say that $N$ and $N'$ \emph{are equal} and write that $N=N'$ if there exists a bijection $\varphi:C\rightarrow C'$ such that $\varphi(\sigma_i(c))=\sigma'_i(\varphi(c))$, for every $1\leq i\leq k$ and $c\in C$.
Graphically, we use different connections to distinguish the edge's type. See the network in Figure~\ref{fig:liftinlayerbackconenecright}.

%As pointed out by Rink and Sanders \cite{RS14}, the previous definition is a particular case of the coupled cell networks defined in  \cite[Definition 2.1]{GST05}. 
%Let $C=\{c_1,\dots,c_n\}$ and denote by $\sigma=[a_1\;\dots\;a_n]$ the function $\sigma:C\rightarrow C$ such that $\sigma(c_j)=a_j$, $1\leq j\leq n$. 

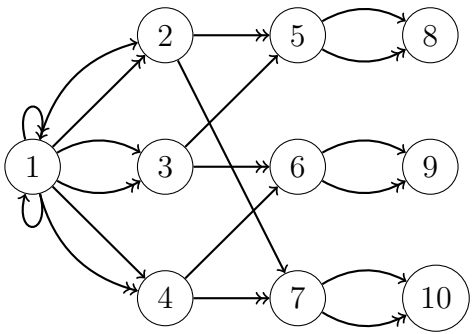
\begin{figure}[h]
\centering
\begin{tikzpicture}
\node (n1) [circle,draw]   {2};
\node (n2) [circle,draw]  [below=of n1] {3};
\node (n3) [circle,draw]  [below=of n2] {4};
\node (n4) [circle,draw] [right=of n1]  {5};
\node (n5) [circle,draw] [below=of n4]  {6};
\node (n6) [circle,draw]  [below=of n5] {7};
\node (n0) [circle,draw]  [left=of n2] {1};
\node (n7) [circle,draw]  [right=of n4] {8};
\node (n8) [circle,draw]  [right=of n5] {9};
\node (n9) [circle,draw]  [right=of n6] {10};

\draw[->, thick] (n0) to  [loop below] (n0);
\draw[->, thick] (n0) to  [bend left] (n1);
\draw[->, thick] (n0) to  [bend left]  (n2);
\draw[->, thick] (n0) to    (n3);
\draw[->, thick] (n2) to   (n4);
\draw[->, thick] (n3) to   (n5);
\draw[->, thick] (n1) to   (n6);
\draw[->, thick] (n4) to  [bend left] (n7);
\draw[->, thick] (n5) to  [bend left] (n8);
\draw[->, thick] (n6) to  [bend left] (n9);

\draw[->>, thick] (n0) to  [loop above] (n0);
\draw[->>, thick] (n0) to   (n1);
\draw[->>, thick] (n0) to  [bend right] (n2);
\draw[->>, thick] (n0) to  [bend right] (n3);
\draw[->>, thick] (n1) to   (n4);
\draw[->>, thick] (n2) to   (n5);
\draw[->>, thick] (n3) to   (n6);
\draw[->>, thick] (n4) to  [bend right] (n7);
\draw[->>, thick] (n5) to  [bend right] (n8);
\draw[->>, thick] (n6) to  [bend right] (n9);
\end{tikzpicture}
\caption{Feed-forward network with $4$ layers}
%represented by the functions $\sigma_1=[1\;1\;1\;1\;3\;4\;2\;5\;6\;7]$ and $\sigma_2=[1\;1\;1\;1\;2\;3\;4\;5\;6\;7]$.}
\label{fig:liftinlayerbackconenecright}
\end{figure}

A network $N$ can be also represented by its adjacency matrices $(A_i)_{i=1}^{k}$. More precisely, 
each matrix $A_i$, $i=1,\dots,k$, is an $|N|\times |N|$ matrix, where the entry $(A_i)_{c\:c'}$ is $1$, if $ c'=\sigma_i(c)$, and $0$, otherwise.

\begin{defi}
Let $N$ be a network represented by the functions $(\sigma_i)_{i=1}^k$. We say that $N$ is a \emph{feed-forward network (FFN)}, if there exists a partition of the set of cells of $N$ into subsets $C_0,C_1,\dots,C_m$ such that $\sigma_i(c)=c$, for every $c\in C_0$, and $\sigma_i(C_j)\subseteq C_{j-1}$, for every $1\leq j\leq m$ and $1\leq i\leq k$.
The subset $C_j$ is called the \emph{$j$th layer} of $N$. 

We assume that every cell not belonging to the last layer, is a source of some edge targeting the next layer, i.e., for every $0\leq j\leq m-1$ and $c\in C_j$ there exists $c'\in C_{j+1}$ and $1 \leq i\leq k$ such that $c=\sigma_i(c')$.
\end{defi}

All feed-forward networks that we consider are connected, i.e., for every two distinct cells $c$ and $c'$ of $N$ there exists a sequence of cells $c_0,c_1,\dots,c_{l-1},c_l$ in $N$ such that $c'=c_0$, $c=c_l$  and there is an edge from $c_{a-1}$ to $c_a$ or an edge from $c_a$ to $c_{a-1}$, for every $1\leq a\leq l$. 
%We call the previous sequence an \emph{undirected path} in $N$ between $c$ and $c'$.

\begin{defi}
We say that a network $N$ is \emph{backward connected for a cell} $c$ if for every cell $c'$ different from $c$ there exists a sequence of cells $c_0,c_1,\dots,c_{l-1},c_l$ in $N$ such that $c'=c_0$, $c=c_l$  and there is an edge from $c_{a-1}$ to $c_a$, for every $1\leq a\leq l$.
The network $N$ is \emph{backward connected} if it is backward connect for some cell.
\end{defi}

%If $N$ is a network represented by $(\sigma_i)_{i=1}^{k}$, then $N$ is backward connected for $c\in C$ if and only if for every $c'\in C\setminus\{c\}$ there exist $1\leq j_1,\dots,j_n\leq k$ such that $$\sigma_{j_1}\circ\dots\circ\sigma_{j_n}(c)=c'.$$

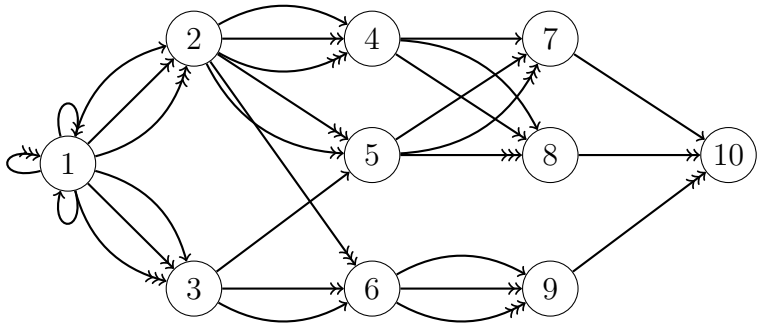
\begin{figure}[h]
\center
\begin{tikzpicture}[node distance=1.6cm]

\node (n1) [circle,draw]   {1};
\node (n2) [circle,draw]  [above right=of n1] {2};
\node (n3) [circle,draw] [below right=of n1]  {3};
\node (n4) [circle,draw] [right=of n2]  {4};
\node (n5) [circle,draw]  [below=of n4,yshift=8mm] {5};
\node (n6) [circle,draw] [right=of n3]  {6};
\node (n7) [circle,draw] [right=of n4]  {7};
\node (n8) [circle,draw] [right=of n5]  {8};
\node (n9) [circle,draw] [right=of n6]  {9};
\node (n10) [circle,draw,label=center:10] [right=of n8]  {\phantom{0}};

\draw[->, thick] (n1) to  [loop below] (n1);
\draw[->, thick] (n1) to  [bend left]  (n2);
\draw[->, thick] (n1) to  [bend left] (n3);
\draw[->, thick] (n2) to  [bend left] (n4);
\draw[->, thick] (n3) to   (n5);
\draw[->, thick] (n3) to  [bend right]  (n6);
\draw[->, thick] (n4) to    (n7);
\draw[->, thick] (n4) to  [bend left] (n8);
\draw[->, thick] (n6) to  [bend left] (n9);
\draw[->, thick] (n7) to  (n10);

\draw[->>, thick] (n1) to  [loop above] (n1);
\draw[->>, thick] (n1) to   (n2);
\draw[->>, thick] (n1) to   (n3);
\draw[->>, thick] (n2) to   (n4);
\draw[->>, thick] (n2) to  [bend right] (n5);
\draw[->>, thick] (n3) to   (n6);
\draw[->>, thick] (n5) to  (n7);
\draw[->>, thick] (n4) to  (n8);
\draw[->>, thick] (n6) to  (n9);
\draw[->>, thick] (n8) to   (n10);

\draw[->>>, thick] (n1) to  [loop left] (n1);
\draw[->>>, thick] (n1) to  [bend right] (n2);
\draw[->>>, thick] (n1) to  [bend right] (n3);
\draw[->>>, thick] (n2) to  [bend right] (n4);
\draw[->>>, thick] (n2) to   (n5);
\draw[->>>, thick] (n2) to   (n6);
\draw[->>>, thick] (n5) to   [bend right] (n7);
\draw[->>>, thick] (n5) to   (n8);
\draw[->>>, thick] (n6) to  [bend right] (n9);
\draw[->>>, thick] (n9) to   (n10);
\end{tikzpicture}
\caption{Backward connected feed-forward network with 5 layers.}
\label{fig:sufcondisnotneccond}
\end{figure}

The network in Figure~\ref{fig:sufcondisnotneccond} is backward connected (for the cell $10$) and the network in Figure~\ref{fig:liftinlayerbackconenecright} is not backward connected.

\begin{obs}
A feed-forward network is backward connected if and only if the cardinality of the last layer is $1$.%$|C_m|=1$.% ( and $s(I(C_j))=C_{j-1}$ for every $1\leq j\leq m$) .
\end{obs}

\section{Lifts of feed-forward networks}\label{sec:liftffn}

We recall now a few facts about balanced colorings, quotient networks and lifts of networks, following \cite{SGP03,GST05,RS14} with emphasis at feed-forward networks. 
%The balanced colorings of a feed-forward network are studied in \cite{ADF17}, we focus on balanced colorings of feed-forward networks that generate feed-forward networks by the quotient operation.
We define two types of basic lifts in feed-forward networks: lifts inside a layer and lifts that create new layers. 
We prove that any lift of a feed-forward network to a backward connected feed-forward network can be obtained by a composition of lifts that create new layers and lifts inside the layers.

Let $N$ be a network represented by the functions $(\sigma_i)_{i=1}^k$. 
A \emph{coloring} of the set of cells of $N$ is an equivalence relation on the set of cells. 
A coloring $\bowtie$ is \emph{balanced} if $\sigma_i(c)\bowtie\sigma_i(c')$, for every $1\leq i \leq k$ and cells $c,c'$ of $N$ such that $c\bowtie c'$. 
It follows from \cite[Proposition 7.2]{RS15} that this definition coincides with the definition of balanced coloring given in \cite[Definition 4.1]{GST05}.
Given a subset of cells $S$ in $N$, we denote by $[S]_{\bowtie}$ the set of $\bowtie$-classes of the cells in $S$, i.e. $[S]_{\bowtie}=\{[c]_{\bowtie}:c\in S\}$. When defining or referring to a coloring, we omit its classes that have only one element.

\begin{defi}[{\cite[Section 5]{GST05}}]
 The \emph{quotient network} of a network $N$, represented by the functions $(\sigma_i:C\rightarrow C)_{i=1}^k$, associated to a balanced coloring $\bowtie$ in $N$ is the network represented by the functions $(\sigma_i^{\bowtie}:[C]_{\bowtie}\rightarrow [C]_{\bowtie})_{i=1}^k$, where
% $[C]_{\bowtie}$ are the equivalence classes of $\bowtie$ on $C$ and
$\sigma_i^{\bowtie}$ is given by  $\sigma_i^{\bowtie}([c]_{\bowtie})=[\sigma_i(c)]_{\bowtie}$, for every $1\leq i\leq k$ and $c\in C$. We denote by $N/\bowtie$ the quotient network of $N$ associated to $\bowtie$.
We also say that a network $L$ is a \emph{lift} of $N$, if $N$ is a quotient of $L$ for some balanced coloring in $L$.
\end{defi}

\begin{obs}
The quotient network of a backward connected network is also backward connected.
\end{obs}

\begin{figure}[h]
\center
\begin{tikzpicture}
\node (n1) [circle,draw]   {2};
\node (n2) [circle,draw]  [below=of n1] {3};
\node (n4) [circle,draw] [right=of n1]  {5};
\node (n5) [circle,draw] [below=of n4]  {6};
\node (n6) [circle,draw]  [below=of n5] {7};
\node (n0) [circle,draw]  [left=of n2] {1};
\node (n7) [circle,draw]  [right=of n4] {8};
\node (n8) [circle,draw]  [right=of n5] {9};
\node (n9) [circle,draw]  [right=of n6] {10};

\draw[->, thick] (n0) to  [bend left] (n1);
\draw[->, thick] (n0) to   (n2);
\draw[->, thick] (n2) to   (n4);
\draw[->, thick] (n2) to  [bend left] (n5);
\draw[->, thick] (n1) to   (n6);
\draw[->, thick] (n0) to  [loop below] (n0);
\draw[->, thick] (n4) to  [bend left] (n7);
\draw[->, thick] (n5) to  [bend left] (n8);
\draw[->, thick] (n6) to  [bend left] (n9);

\draw[->>, thick] (n0) to  (n1);
\draw[->>, thick] (n0) to  [bend right] (n2);
\draw[->>, thick] (n1) to   (n4);
\draw[->>, thick] (n2) to  [bend right] (n5);
\draw[->>, thick] (n2) to   (n6);
\draw[->>, thick] (n0) to  [loop above] (n0);
\draw[->>, thick] (n4) to  [bend right] (n7);
\draw[->>, thick] (n5) to  [bend right] (n8);
\draw[->>, thick] (n6) to  [bend right] (n9);
\end{tikzpicture}
\caption{Quotient of the network in Figure~\ref{fig:liftinlayerbackconenecright} associated with the balanced coloring given by the class $\{3,4\}$.}
%The network on the right, $L$, is a lift of the network on the left, $N$.  Consider a function $f\in\mathcal{V}_0(L)$, we have that every bifurcation branch of $f$ in $L$ is lifted from $N$. Note that $L$ is not backward connected.}
\label{fig:liftinlayerbackconenecleft}
\end{figure}
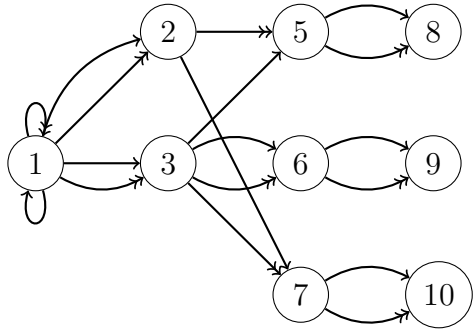

\begin{exe}
Consider the network $N$ in Figure~\ref{fig:liftinlayerbackconenecright} represented by the functions $(\sigma_1,\sigma_2)$ and the coloring $\bowtie$, where $\sigma_1(i)=\sigma_2(i)=1$, if $i=1,2,3,4$, $\sigma_1(7)=\sigma_2(5)=2$, $\sigma_1(5)=\sigma_2(6)=3$, $\sigma_1(6)=\sigma_2(7)=4$ and $\sigma_1(i)=\sigma_2(i)=i-3$, if $i=8,9,10$, and $\bowtie$ is given by $3\bowtie 4$. The coloring $\bowtie$ is balanced, since $\sigma_i(3)\bowtie \sigma_i(4)$, for $i=1,2$.
The network in Figure~\ref{fig:liftinlayerbackconenecleft} is the quotient of $N$ associated to $\bowtie$.
\end{exe}

In \cite{ADF17}, the authors studied and described the balanced colorings of feed-forward networks.
The set of balanced colorings forms a partially order set as studied in \cite{S07} given by the refinement relation. 
Given two balanced colorings $\bowtie', \bowtie$ of a network $N$, we say that $\bowtie'$ refines $\bowtie$ and we write $\bowtie'\preceq \bowtie$, if $c\bowtie' d$ implies that $c\bowtie d$, for every cells $c$ and $d$ of $N$. 
We have that if $\bowtie'\preceq \bowtie$, then $N/\bowtie$ is a quotient of $N/\bowtie'$.

If $L$, $N$ and $Q$ are networks such that $L$ is a lift of $N$ and $N$ is a lift of $Q$, then $L$ is a lift of $Q$.
Moreover, we say that the lift of $Q$ to $L$ is given by the composition of the lift of $Q$ to $N$ and the lift of $N$ to $L$.
%In some cases, given a lift $L$ of a network $Q$, it is possible to find an intermediate network $N$ which is simultaneously a quotient of $L$ and a lift of $Q$, see \cite[Theorem 2.4]{DM17}. 
In some cases, a lift can be seen as the composition of two lifts, see \cite[Theorem 2.4]{DM17}. 
In the next result, we give a sufficient condition for that to occur. 
%In the next result, we give a condition for the existence of such intermediate network. 

\begin{lem}\label{splitstep}
Let $L$ be a network represented by the functions $(\sigma_i:C\rightarrow C)_{i=0}^k$, $\bowtie$ a balanced coloring in $L$ and $S\subseteq C$ such that $\sigma_i(S)\subseteq S$, for $1\leq i\leq k$.
%$[C\setminus S]_{\bowtie}\cap [S]_{\bowtie}=\emptyset$ and $\sigma^{\bowtie}_i([S]_{\bowtie})\subseteq [S]_{\bowtie}$, for $1\leq i\leq k$.
Then, there exists a balanced coloring $\bowtie'$ in $L$ such that $L/\bowtie$ is a quotient of $L/\bowtie'$,
 $[C\setminus S]_{\bowtie'}= C\setminus S$ and there exists a bijection between $[S]_{\bowtie}$ and $[S]_{\bowtie'}$.
\end{lem}

\begin{proof}
Let $L$ be a network represented by the functions $(\sigma_i:C\rightarrow C)_{i=0}^k$, $\bowtie$ a balanced coloring in $L$ and $S\subseteq C$ such that 
$\sigma_i(S)\subseteq S$, for $1\leq i\leq k$.

Define $\bowtie'$ as the coloring of $L$ such that $c\bowtie' c'$ if $c\bowtie c'$ and $c,c'\in S$. % and $[C\setminus S]_{\bowtie'}= C\setminus S$.
Let $c,c'\in S$ such that $c\bowtie' c'$. 
Then $c\bowtie c'$, $\sigma_i(c)\bowtie\sigma_i(c')$ and $\sigma_i(c),\sigma_i(c')\in S$, for every $1\leq i\leq k$.
%Moreover $\sigma_i(c),\sigma_i(c')\in S$, since $c,c'\in S$ and $\sigma_i(S)\subseteq S$, for $1\leq i\leq k$.
Hence $\sigma_i(c)\bowtie'\sigma_i(c')$, for every $1\leq i\leq k$, and  $\bowtie'$ is a balanced coloring of $L$.
Note that $\bowtie'\preceq \bowtie$ and so $L/\bowtie$ is a quotient of $L/\bowtie'$.

The $\bowtie'$-class of any cell in $C\setminus S$ is singular, so $[C\setminus S]_{\bowtie'}= C\setminus S$.

Let $\alpha:[S]_{\bowtie} \rightarrow [S]_{\bowtie'}$ be given by $\alpha([c]_{\bowtie})=[c]_{\bowtie'}$, where $c\in S$.
Let $c,c'\in S$ such that $[c]_{\bowtie}=[c']_{\bowtie}$. 
Then $c\bowtie c'$ and $c\bowtie' c'$. 
So $\alpha$ is well-defined.
Suppose that  $\alpha([c]_{\bowtie})= \alpha([c']_{\bowtie})$.
Then $c\bowtie' c'$ and $c\bowtie c'$.
So $[c]_{\bowtie}=[c']_{\bowtie}$ and $\alpha$ is injective.
Let $[c]_{\bowtie'}\in [S]_{\bowtie'}$.
Then $\alpha([c]_{\bowtie})=[c]_{\bowtie'}$ and $\alpha$ is surjective. 
So $\alpha$ is a bijection between $[S]_{\bowtie}$ and $[S]_{\bowtie'}$. 
\end{proof}

Next, we define two basic lifts in feed-forward networks.

\begin{defi}
Let $N$ be a feed-forward network and $L$ a feed-forward lift of $N$. Denote the layers of $N$ and $L$ by $C_0,C_1,\dots,C_m$ and $C'_0,C'_1,\dots,C'_n$, respectively.

We say that $L$ is a \emph{lift inside the layer $C_j$}, where $ 0\leq j\leq m$, if $m=n$, $|C'_j|\neq |C_j|$ and $|C'_i|=|C_i|$ for every $i\neq j$.

We say that $L$ is a \emph{lift that creates $n-m$ new layers}, if $m<n$, $|C'_0|=|C'_1|=\dots=|C'_{n-m}|=|C_0|$ and $|C'_{n-m+j}|=|C_j|$, for every $1 \leq j \leq m$.
\end{defi}

The network in Figure~\ref{fig:liftinlayerbackconenecright} is a lift inside the second layer of the network in Figure~\ref{fig:liftinlayerbackconenecleft}.
In the following example, we apply Lemma~\ref{splitstep} and see that some lifts of feed-forward networks are the composition of lifts inside a layer.

\begin{exe}\label{exe:compliftinslay}
Let $N$ be a feed-forward network and $L$ a feed-forward lift of $N$. 
Suppose that $N$ and $L$ have the same number of layers and denote the layers of $L$ by $C_0,C_1,\dots,C_m$. 

Consider the set $S_{m-1}=C_0 \cup C_1 \cup \dots \cup C_{m-1}$ and apply Lemma~\ref{splitstep} to the lift $L$ of $N$, then there exists a network $Q_{m-1}$ such that the lift of $N$ to $L$ is the composition of the lift of $N$ to $Q_{m-1}$ and the lift of $Q_{m-1}$ to $L$.
It is not hard to prove that  $Q_{m-1}$ is also a feed-forward network with the same number of layers than $N$ and $L$. 
By Lemma~\ref{splitstep}, the layers of $Q_{m-1}$ and $N$ have the same number of cells, except the last one. 
Hence $Q_{m-1}$ is a lift of $N$ inside a layer (or $N=Q_{m-1}$).

We can repeat the previous process by considering the set $S_i=C_0 \cup C_1 \cup \dots \cup C_i$ and applying Lemma~\ref{splitstep} to the lift $L$ of $Q_{i+1}$, for each $m-2\geq i\geq 1$. 
In this way, we obtain a sequence of networks $Q_{m},Q_{m-1}, \dots, Q_1, Q_0$ such that $Q_{j+1}$ is a lift of $Q_j$ inside a layer (or $Q_{j+1}=Q_{j}$), where $j=0,\dots,m-2$, $L=Q_{m}$ and $N=Q_0$.
Therefore the lift of $N$ to $L$ is the composition of lifts inside the layers and a lift from $Q_{m-1}$ to $L$.

If $L$ is backward connected, then $Q_{m-1}=L$ and the lift of $N$ to $L$ is the composition of lifts inside the layers.
\end{exe}

%Using Lemma~\ref{splitstep}, we have a method to obtain a lifted network by steps.

%\subsection{Lifts of Feed-Forward Networks}

 According to the results obtained in \cite{ADF17}, we give an auxiliary lemma.

\begin{lem}\label{layliftfeed}
Let $N$ be a feed-forward network with layers $C_0,C_1,\dots,C_m$ and $L$ a feed-forward lift of $N$ with layers $C'_0,C'_1,\dots,C'_n$ such that $L$ is a backward connected. Denote by $\bowtie$ the balanced coloring of $L$ such that $L/\bowtie=N$.
Then
$$[C'_{n-m}]_{\bowtie}=\dots=[C'_{0}]_{\bowtie}=C_{0}\quad \quad \quad [C'_{n-j}]_{\bowtie}=C_{m-j},$$
for $0\leq j\leq m-1$. 
\end{lem}

\begin{proof}
Let $N$ be a feed-forward network with layers $C_0,C_1,\dots,C_m$ and $L$ a feed-forward lift of $N$ with layers $C'_0,C'_1,\dots,C'_n$ such that $L$ is a backward connected.
Assume that $N$ and $L$ are represented by $(\sigma_i)_{i=0}^k$ and $(\sigma'_i)_{i=0}^k$, respectively.
Let $\bowtie$ be a balanced coloring such that $L/\bowtie=N$.
Then $n\geq m$ and $[\sigma'_i(c)]_{\bowtie}=\sigma_i([c]_{\bowtie})$, for every cell $c$ in $L$ and $1\leq i\leq k$.

Since $L$ is backward connected, $N$ is also backward connected, $C'_n=\{c_1\}$, $C_m=\{[c_1]_{\bowtie}\}$ and $[C'_{n}]_{\bowtie}=C_{m}$.
If $m=0$, then $N$ has only one cell and there is only one equivalence class of $\bowtie$. 
Hence $[C'_{n}]_{\bowtie}=\dots=[C'_{0}]_{\bowtie}=C_{0}$. 
%If $n=0$, then $m=0$ and the result holds.

%The cells $c_1$ and $[c_1]_{\bowtie}$ are the unique cells of the last layer of $L$ and $N$, respectively. %Clearly $c_1\in [c_2]_{\bowtie}$ and  $[c_1]_{\bowtie}= [c_2]_{\bowtie}$.
%If $c_2 \neq c_1$, then there exists a cell $c_3$ in $L$ and $1\leq i\leq k$ such that $c_2=\sigma'_i(c_3)$. So $[c_2]_{\bowtie}=\sigma_i([c_3]_{\bowtie})$ and $[c_2]_{\bowtie}$ is a source cell of an edge. So $c_2=c_1$ and 
%Thus, $[C'_{n}]_{\bowtie}=C_{m}$.

Now, suppose that $m>0$. 
Let $d_1\in C_{m-j}$ where $1\leq j\leq m$. Then there exist $1\leq i\leq k$ and $d_2\in C_{m+1-j}$ such that $d_1=\sigma_i(d_2)$.  Assuming that $[C'_{n+1-j}]_{\bowtie}=C_{m+1-j}$, there exists $d'_2\in C'_{n+1-j}$ such that $d_2=[d'_2]_{\bowtie}$, $d_1=[\sigma'_i(d'_2)]_{\bowtie}$ and $\sigma'_i(d'_2)\in C'_{n-j}$. Thus $C_{m-j}\subseteq [C'_{n-j}]_{\bowtie}$. On the other hand, let $d'_1\in C'_{n-j}$. Then there exist $1\leq i\leq k$ and $d'_2\in C'_{n+1-j}$ such that $d'_1=\sigma'_i(d'_2)$. Assuming that $[C'_{n+1-j}]_{\bowtie}=C_{m+1-j}$, we have that $[d'_2]_{\bowtie}\in C_{m+1-j}$ and $[\sigma'_i(d'_2)]_{\bowtie}\in C_{m-j}$. Therefore
$[C'_{n-j}]_{\bowtie}=C_{m-j}$.
Since $[C'_{n}]_{\bowtie}=C_{m}$, 
$$[C'_{n-j}]_{\bowtie}=C_{m-j}\quad \quad 0\leq j\leq m.$$

In particular, $[C'_{n-m}]_{\bowtie}=C_{0}$. From this and the fact that $\sigma_i(C_0)=C_0$, for $1\leq i\leq k$, using the same argument as above we conclude that
\[[C'_{n-m}]_{\bowtie}=\dots=[C'_{0}]_{\bowtie}=C_{0}.\qedhere\]
\end{proof}

Using the Example~\ref{exe:compliftinslay} and Lemma~\ref{layliftfeed}, we see how to decompose a feed-forward lift into lifts that create new layers and lifts inside a layers.

\begin{prop}\label{prop:compolift}
Let $N$ be a feed-forward network and $L$ a feed-forward lift of $N$ such that $L$ is backward connected. 
Then, the lift of $N$ to $L$ is the composition of a lift that creates new layers with lifts inside the layers.
%\begin{itemize}
	%\item Lift that creates $n-m$ new layers;
	%\item Lift inside $C_{m-1}$;
	%\item Lift inside $C_{m-2}$;
		%\item[] \hspace{2cm} $\vdots$
	%\item Lift inside $C_{1}$;
	%\item $n-m$ Lifts inside each new layer;
	%\item Lift inside $C_{0}$.
	%\end{itemize}
\end{prop}

\begin{proof}
Let $N$ be a feed-forward network with layers $C_0,C_1,\dots,C_m$ and $L$ a feed-forward lift of $N$ with layers $C'_0,C'_1,\dots,C'_n$ such that $L$ is a backward connected and represented by $(\sigma_i)_{i=1}^k$.  
Denote by $\bowtie$ the balanced coloring in $L$ such that $N$ is the quotient network of $L$ associated to $\bowtie$.
%Denote by  the functions that represent $L$.

%We show that there exists a quotient network $Q$ of $L$ such that $Q$ is a feed-forward network, $Q$ has the same number of layers as $L$ and $Q$ is a lift of $N$ that creates $n-m$ new layers. 

 Define the coloring $\bowtie_1$ in $L$ such that $c\bowtie_1 d$ if $c\bowtie d$ and $c,d \in C'_j$ for $0\leq j\leq n$.
Let $c\bowtie_1 d$. 
Since $\bowtie$ is balanced, we have that $\sigma_i(c)\bowtie \sigma_i(d)$ and $\sigma_i(c),\sigma_i(d) \in C'_{j'}$, where $j'=\max\{0,j-1\}$ and $0\leq i\leq k$. Then $\sigma_i(c)\bowtie_1\sigma_i(d)$ for $0\leq i\leq k$. Hence $\bowtie_1$ is balanced. 

Define the network $Q_1=L/\bowtie_1$ and $A_j=[C'_j]_{\bowtie_1}$, for $0\leq j \leq n$.
Let $c\in A_0$. There exists $d\in C'_0$ such that $c=[d]_{\bowtie_1}$. 
Then $\sigma^{\bowtie_1}_i(c)=\sigma^{\bowtie_1}_i([d]_{\bowtie_1})= [\sigma_i(d)]_{\bowtie_1}=[d]_{\bowtie_1}=c$, for $1\leq i\leq k$. 
Let $c\in A_j$ and $1\leq j\leq m$. 
There exists $d\in C'_j$ such that $c=[d]_{\bowtie_1}$. 
Since $\sigma_i(d)\in C'_{j-1}$, $\sigma^{\bowtie_1}_i(c)= [\sigma_i(d)]_{\bowtie_1}\in A_{j-1}$, for $1\leq i\leq k$.
If $j<m$, then there exist $d'\in C'_{j+1}$ and $1 \leq i\leq k$ such that $d=\sigma_i(d')$ and $c= \sigma^{\bowtie_1}_i([d']_{\bowtie_1})$.
Therefore $Q_1$ is a feed-forward network with layers $A_0,A_1,\dots,A_n$.

Note that $\bowtie_1\preceq \bowtie$. Hence $Q_1$ is a lift of $N$, if $\bowtie_1\prec \bowtie$, and $Q_1=N$, if $\bowtie_1=\bowtie$. It follows from Lemma~\ref{layliftfeed} that 
$$|A_{n-m+j}|=|[C'_{n-m+j}]_{\bowtie_1}|=|[C'_{n-m+j}]_{\bowtie}|=|C_{j}|,\quad \quad 1 \leq j \leq m,$$
 and 
$$ |A_{j'}|=|[C'_{j'}]_{\bowtie_1}|=|[C'_{j'}]_{\bowtie}|=|C_{0}|,\quad \quad  0 \leq j' \leq n-m.$$
Hence $Q_1$ is a lift of $N$ that creates $n-m$ new layers or $Q_1=N$.

The networks $Q_1$ and its lift $L$ have the same number of layers. 
Following Example~\ref{exe:compliftinslay}, we see that the lift of $Q_1$ to $L$ is the composition of lifts inside the layers. 
Therefore the lift of $N$ to $L$ can be obtained by the composition of a lift that creates new layer with lifts inside a layer. 
%Define $S_j=\bigcup_{l=0}^{n-j} C_l'$, for $2\leq j \leq n$. Then $\sigma_i(S_j)\subset S_j$, for every $1\leq i\leq k$.  
%Let $Q_j$ be the intermediate network obtained in Lemma~\ref{splitstep} by considering the set $S_j$, for $2\leq j \leq n$. 
%By the construction used in the proof of Lemma~\ref{splitstep}, we see that $Q_j$ is a feed-forward network and $Q_j$ is a lift of $Q_{j-1}$ inside the layer $A_{n+1-j}$ or $Q_j=Q_{j-1}$, for $2\leq j \leq n$, and $L$ is a lift of $Q_n$ inside the layer $A_{0}$ or $L=Q_{n}$.
%Therefore $L$ can be obtained from $N$ by the composition of lifts that create new layers and lifts inside some layers.
\end{proof}

The condition imposed in Proposition~\ref{prop:compolift} that the lift network is backward connected can not be removed as the next example shows.

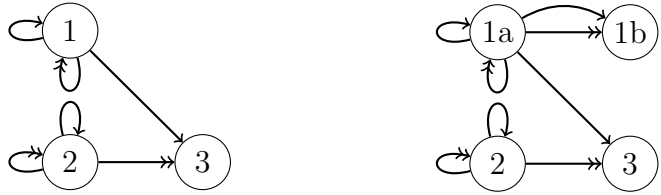
\begin{figure}[h]
\center
\begin{subfigure}{0.4\textwidth}
\center
\begin{tikzpicture}
\node (n1) [circle,draw]   {1};
\node (n2) [circle,draw]  [below=of n1] {2};
\node (n3) [circle,draw] [right=of n2]  {3};

\draw[->, thick] (n1) to  [loop left] (n1);
\draw[->, thick] (n2) to  [loop above] (n2);
\draw[->, thick] (n1) to   (n3);

\draw[->>, thick] (n1) to  [loop below] (n1);
\draw[->>, thick] (n2) to  [loop left] (n2);
\draw[->>, thick] (n2) to   (n3);
\end{tikzpicture}
%\caption{ }
%\label{fig:bifbranotlift}
\end{subfigure}
\begin{subfigure}{0.4\textwidth}
\centering
\begin{tikzpicture}
\node (n1) [circle,draw, label=center:1a]   {\phantom{0}};
\node (n4) [circle,draw, label=center:1b] [right=of n1]   {\phantom{0}};
\node (n2) [circle,draw]  [below=of n1] {2};
\node (n3) [circle,draw] [right=of n2]  {3};

\draw[->, thick] (n1) to  [loop left] (n1);
\draw[->, thick] (n1) to  [bend left] (n4);
\draw[->, thick] (n2) to  [loop above] (n2);
\draw[->, thick] (n1) to   (n3);

\draw[->>, thick] (n1) to  [loop below] (n1);
\draw[->>, thick] (n1) to   (n4);
\draw[->>, thick] (n2) to  [loop left] (n2);
\draw[->>, thick] (n2) to   (n3);
\end{tikzpicture}
%\caption{ }
\end{subfigure}
\caption{The feed-forward network on the right is a lift of the feed-forward network on the left and it is not the composition of lifts that create new layers and lifts inside layers. Note that the lift network is not backward connected.}
\label{fig:feedforwliftnotdivi2}
\end{figure}

\begin{exe}\label{exe:feedforwliftnotdivi2}
Let $N$ be the feed-forward network on the left of Figure~\ref{fig:feedforwliftnotdivi2} and $L$ the feed-forward network on the right of Figure~\ref{fig:feedforwliftnotdivi2}. The network $L$ is a lift of $N$, considering the coloring in $L$ given by the class $\{1a,1b\}$. This lift cannot be obtained by a composition of lifts that create new layers and lifts inside the layers. Note that $N$ and $L$ have the same number of layers and the coloring in $L$ given by the class $\{1b,3\}$ is not balanced. However, $L$ is not backward connected.
\end{exe}

The lifts inside a layer can be further decomposed using splits.

\begin{defi}
Let $N$ be a network  and $L$ a lift of $N$.
We say that $L$ is the \emph{split} of a cell $c$ in $N$ into cells $c_1, c_2,\dots,c_l$ in $L$, if the coloring $\bowtie$ in $L$ given by $c_i\bowtie c_j$, for $1\leq i,j\leq l$, is balanced, $L/\bowtie=N$ and $[c_i]_{\bowtie}=c$. 
\end{defi}

The network in Figure~\ref{fig:liftinlayerbackconenecright} is a split of the cell $3$ in Figure~\ref{fig:liftinlayerbackconenecleft} into the cells $3$ and $4$.
The network in the right of Figure~\ref{fig:feedforwliftnotdivi2} is a split of the cell $1$ in the left of Figure~\ref{fig:feedforwliftnotdivi2} into the cells $1a$ and $1b$. 

\begin{obs}\label{obs:liftinslaycompsplit}
%Let $N$ be a feed-forward network and $L$ is a lift inside a layer.
By Lemma~\ref{splitstep}, a lift inside a layer is the composition of splits of a cell into two cells. 
\end{obs}

Next, we prove that if a feed-forward network is backward connected and a lift of some feed-forward network, then there is an unique balanced coloring associated to the lift.

\begin{lem}\label{lem:unibalcolbackcon}
Let $N$ be a feed-forward network and $L$ a lift of $N$ such that $L$ is a backward connected feed-forward network. Let $\bowtie_1,\bowtie_2$ be colorings in $L$. If $L/\bowtie_1=L/\bowtie_2=N$, then $\bowtie_1=\bowtie_2$.
\end{lem}

\begin{proof}
Let $N$ be a feed-forward network and $L$ a lift of $N$ such that $L$ is a backward connected feed-forward network. Let $\bowtie_1,\bowtie_2$ be colorings in $L$. Denote by $C_0,\dots,C_m$ the layers of $N$, by $(\sigma_i^N)_{i=1}^k$ the representative functions of $N$, by $C'_0,C'_1,\dots,C'_n$ the layers of $L$ and by $(\sigma_i^L)_{i=1}^k$ the representative functions of $L$. Suppose that $L/\bowtie_1=L/\bowtie_2=N$.

Since $L$ is backward connected, we know that $N$ is backward connected and $|C_m|=|C'_n|=1$. So for $c\in C_{n}'$, we have that $[c]_{\bowtie_1}=[c]_{\bowtie_2}$. Next, we prove that if for every $c\in C_j'$, $j>0$, $[c]_{\bowtie_1}=[c]_{\bowtie_2}$, then for every $d\in C'_{j-1}$ we have that $[d]_{\bowtie_1}=[d]_{\bowtie_2}$. Suppose for every $c\in C_j'$, $j>0$, that $[c]_{\bowtie_1}=[c]_{\bowtie_2}$. Let $d\in C_{j-1}'$. Then there exist $1\leq i\leq k$ and $c\in C_j'$ such that $\sigma_i^L(c)=d$. Thus 
$$[d]_{\bowtie_1}=[\sigma_i^L(c)]_{\bowtie_1}=\sigma_i^N([c]_{\bowtie_1})=\sigma_i^N([c]_{\bowtie_2})=[\sigma_i^L(c)]_{\bowtie_2}=[d]_{\bowtie_2}.$$
By induction, for every $0\leq j\leq n$ and $c\in C_j'$ we have that $[c]_{\bowtie_1}=[c]_{\bowtie_2}$.
Hence $\bowtie_1=\bowtie_2$. 
\end{proof}

In the next example, we see that the previous result does not hold if the lift is not backward connected.

\begin{figure}[h]
\center
\begin{subfigure}[t]{0.45\textwidth}
\center
\begin{tikzpicture}
\node (n1) [circle,draw]   {1};
\node (n2) [circle,draw]  [below=of n1] {2};
\node (n4) [circle,draw] [right=of n1]  {4};
\node (n5) [circle,draw] [below=of n4]  {5};
\node (n6) [circle,draw]  [below=of n5] {6};

\draw[->, thick] (n1) to  [loop below] (n1);
\draw[->, thick] (n2) to  [loop below]  (n2);
\draw[->, thick] (n2) to   (n4);
\draw[->, thick] (n2) to  [bend left] (n5);
\draw[->, thick] (n1) to   (n6);

\draw[->>, thick] (n1) to  [loop above] (n1);
\draw[->>, thick] (n2) to  [loop above] (n2);
\draw[->>, thick] (n1) to   (n4);
\draw[->>, thick] (n2) to  [bend right] (n5);
\draw[->>, thick] (n2) to   (n6);
\end{tikzpicture}
%\caption{}
\end{subfigure}
\begin{subfigure}[t]{0.45\textwidth}
\centering
\begin{tikzpicture}
\node (n1) [circle,draw]   {1};
\node (n2) [circle,draw]  [below=of n1] {2};
\node (n3) [circle,draw]  [below=of n2] {3};
\node (n4) [circle,draw] [right=of n1]  {4};
\node (n5) [circle,draw] [below=of n4]  {5};
\node (n6) [circle,draw]  [below=of n5] {6};

\draw[->, thick] (n1) to  [loop below] (n1);
\draw[->, thick] (n2) to  [loop below]  (n2);
\draw[->, thick] (n3) to  [loop below]  (n3);
\draw[->, thick] (n2) to   (n4);
\draw[->, thick] (n3) to   (n5);
\draw[->, thick] (n1) to   (n6);

\draw[->>, thick] (n1) to  [loop above] (n1);
\draw[->>, thick] (n2) to  [loop above] (n2);
\draw[->>, thick] (n3) to  [loop above] (n3);
\draw[->>, thick] (n1) to   (n4);
\draw[->>, thick] (n2) to   (n5);
\draw[->>, thick] (n3) to   (n6);
\end{tikzpicture}
%\caption{}
\end{subfigure}
\caption{The network on the right, $L$, is a lift of the network on the left, $N$. There are $3$ different balanced colorings in $L$ such that $N$ is a quotient network of $L$ associated to each of those colorings.}
\label{fig:uniqbalcolback}
\end{figure}
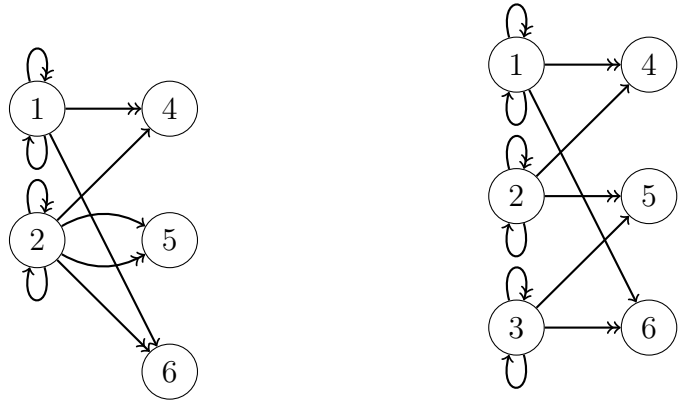

\begin{exe}\label{exe:uniqbalcolback}
Let $N$ be the feed-forward network on the left of Figure~\ref{fig:uniqbalcolback} and $L$ the feed-forward network on the right of Figure~\ref{fig:uniqbalcolback}. Consider the balanced colorings in $L$: $\bowtie_1$ given by $1\bowtie_1 2$; $\bowtie_2$ given by $2\bowtie_2 3$; and $\bowtie_3$ given by $1\bowtie_3 3$. Then $N=L/\bowtie_1=L/\bowtie_2=L/\bowtie_3$ and $L$ is a lift inside a layer. 
\end{exe}

\section{Feed-forward systems}\label{sec:ffs}

In this section, we recall the concept of coupled cell systems associated to a network, synchrony subspace and steady-state bifurcations, following \cite{SGP03,GST05}. 

%\subsection{Coupled cell systems}\label{sec:ccs}

Let $N$ be a network represented by the functions $(\sigma_i)_{i=1}^k$. 
For each cell $c$ of the network, we associate a coordinate $x_c\in \mathbb{R}$. 
We say that $F:\mathbb{R}^{|N|} \rightarrow \mathbb{R}^{|N|}$ is an \emph{admissible vector field} for $N$, if there is $f:\mathbb{R}\times \mathbb{R}^{k}\rightarrow \mathbb{R}$ such that
$$(F(x))_c=f(x_c, x_{\sigma_1(c)},\dots,x_{\sigma_k(c)}),$$
for every cell $c$ of $N$. 
The admissible vector fields for $N$ are defined by the functions $f:\mathbb{R}\times \mathbb{R}^{k}\rightarrow \mathbb{R}$. 
We denote by $f^N$ the admissible vector field for $N$ defined by $f$.

A \emph{coupled cell system} associated to a network $N$ is a system of ordinary differential equations 
$$\dot{x}=f^N(x),\quad x\in \mathbb{R}^{|N|},$$
where $f^N : \mathbb{R}^{|N|}\rightarrow \mathbb{R}^{|N|}$ is an admissible vector field for $N$.
When $N$ is a feed-forward network, we refer to a coupled cell system associated to $N$ as a \emph{feed-forward system}.

Let $(A_i)_{i=1}^{k}$ be the adjacency matrices of $N$. If the function $f$ is differentiable at the origin, then the Jacobian matrix of $f^N$ at the origin is
$$J_f^{N}:= (Df^{N})_{0}=f_0Id + \sum_{i=1}^{k} f_i A_i,$$
where $Id$ is the identity $|N|\times|N|$-matrix and 
$$f_i:=\frac{\partial f}{\partial x_{i}}(0,0\dots,0),\quad \quad 0\leq i\leq k.$$

Let $C_0,C_1,\dots,C_m$ be the layer of $N$.
For every feed-forward system $f^N$ associated to $N$, the Jacobian matrix at the origin has the form
$$J^N_f=\begin{bmatrix}
(\sum_{i=0}^k f_j) Id_0&0&0&\dots&0&0\\
R_1&f_0Id_1&0&\dots&0&0\\
0&R_2&f_0Id_2&\dots&0&0\\
\vdots&\vdots&\vdots&\ddots&\vdots&\vdots\\
0&0&0&\dots&f_0Id_{m-1}&0\\
0&0&0&\dots&R_m&f_0Id_m
\end{bmatrix},$$
where $Id_i$ is the identity matrix of size $|C_i|$, $i=0,1,\dots,m$, $R_j$ is a $|C_{j}|\times|C_{j-1}|$-matrix, $j=1,\dots,m$. 
The eigenvalues of $J_f^N$ are $\sum_{i=0}^k f_j$ and $f_0$.

%\subsection{Synchrony subspaces}

%Let $N$ be a network. 
A \emph{polydiagonal subspace} is a subspace of $\mathbb{R}^{|N|}$ given by the equalities of some cell coordinates. Given a coloring $\bowtie$ on the set of cell of $N$, the polydiagonal subspace associated to $\bowtie$ is 
$$\Delta_{\bowtie}:=\left\{x:  c\bowtie d \Rightarrow x_c=x_d \right\}\subseteq \mathbb{R}^{|N|}.$$
And any polydiagonal subspace of $\mathbb{R}^{|N|}$ defines an unique coloring on the set of cell of $N$.

Given a function $G:\mathbb{R}^{|N|}\rightarrow \mathbb{R}^{|N|}$ and a subset $\Delta \subseteq \mathbb{R}^{|N|}$, we say that $\Delta$ is \emph{invariant} by $G$ if $G(\Delta)\subseteq \Delta$.
A \emph{synchrony subspace} of a network $N$ is a polydiagonal subspace of $\mathbb{R}^{|N|}$ that is invariant by any admissible vector field of $N$. 
There is an one to one correspondence between balanced colorings $\bowtie$ and synchrony subspaces $\Delta_{\bowtie}$. See \cite[Theorem 4.3]{GST05}. More specifically, the polydiagonal $\Delta_{\bowtie}$ associated to a coloring $\bowtie$ is a synchrony subspace of $N$ if and only if the coloring $\bowtie$ is balanced.

Since a synchrony subspace $\Delta_{\bowtie}$ is invariant by every admissible vector field $f^N$ of $N$, every coupled cell system of $N$ given by $f^N$ can be restricted to $\Delta_{\bowtie}$. Each restricted system is a coupled cell system of $N/\bowtie$ given by $f^{N/\bowtie}$. Moreover, given a solution $y(t)\in\mathbb{R}^{|N/\bowtie|}$ of the coupled cell system of $N/\bowtie$ given by $f^{N/\bowtie}$, we have that $x(t)=(x_c(t))$, where $x_c(t)=y_{[c]_{\bowtie}}(t)$ is a solution of the coupled cell system of $N$ given by $f^{N}$. See \cite[Theorem 5.2]{GST05}.

%\subsection{Steady-state bifurcations on coupled cell systems}

Let $G:\mathbb{R}^{d}\times\mathbb{R}\rightarrow \mathbb{R}^{d}$ be a family of smooth vector fields, $d>0$ and the corresponding dynamical systems, depending on the parameter $\lambda$,
\begin{equation}\label{eq:gends}
\dot{x}=G(x, \lambda).
\end{equation}
Consider an equilibrium $(x^*,\lambda^*)$ of $(\ref{eq:gends})$, i.e.  $G(x^*,\lambda^*)=0$.
The family of dynamical systems $(\ref{eq:gends})$ suffers a \emph{local bifurcation} at $(x^*,\lambda^*)$ if for every neighborhoods $U_x$ and $U_\lambda$ of $x^*$ and $\lambda^*$, respectively, there exists $\lambda_1,\lambda_2\in U_{\lambda}$ such that the family $(\ref{eq:gends})$ at $\lambda_1$ and $\lambda_2$ have different topological structures (different stability/number of equilibrium points or periodic orbits, etc.).
A necessary condition for a local bifurcation is that the Jacobian of $G$ at $(x^*,\lambda^*)$, $DG_{(x^*,\lambda^*)}$, has an eigenvalue with zero real part.
We focus on steady-state bifurcations and we say that a \emph{steady-state bifurcation} at $(x^*,\lambda^*)$ occurs if the number of equilibrium points in a neighborhood of $x^*$ changes when the parameter $\lambda$ crosses $\lambda^*$. 
A necessary condition for the occurrence of a steady-state bifurcation at $(x^*,\lambda^*)$ is that $0$ is an eigenvalue of $DG_{(x^*,\lambda^*)}$.

%Let $N$ be a network represented by the functions $(\sigma_{i})_{i=1}^{k}$. 
In order to study the steady-state bifurcations of a family coupled cell systems associated to $N$ from a fully synchronous equilibrium at $\lambda=0$, we consider a family of smooth functions $f:\mathbb{R}\times\mathbb{R}^{k}\times\mathbb{R}\rightarrow \mathbb{R}$ such that $$f(0,0,\dots,0,\lambda)=0,$$
for every $\lambda\in\mathbb{R}$. We denote by $\mathcal{V}(N)$ the set of those functions.
%In Section~\ref{sec:ccs} we saw that, if $N$ is a feed-forward network, then the eigenvalues of $J^N_f$ are $\sum_{i=0}^k f_j$ and $f_0$. 
The set of functions $f\in \mathcal{V}(N)$ such that a steady-state bifurcation occurs at $(0,0)$ for $f^N$ is given by the union of the following sets:
$$\mathcal{V}_k(N):=\{f\in \mathcal{V}(N): \sum_{i=0}^k f_j=0\},\quad \quad\mathcal{V}_0(N):=\{f\in \mathcal{V}(N): f_0=0\}.$$

Thus $\mathcal{V}_k(N)$ denotes the set of functions with a bifurcation condition associated with the \emph{valency} of $N$ and $\mathcal{V}_0(N)$ the set of functions with a bifurcation condition associated with the \emph{internal dynamics} of the cells.

Next, we define equilibrium branches of a coupled cell system associated to $N$ given by $f\in \mathcal{V}(N)$. 
%We see, in the next section, that the equilibrium branches can be defined only for positive or negative values of $\lambda$. 
We say that $D\subseteq \mathbb{R}$ is a \emph{domain} if 
%$0\in D$ and there exists an open neighborhood $U$ of $0$ such that $(D\cap U) \setminus \{0\}$ is an union of open sets. 
%Usually, the domain 
$D$ has one of the following forms: $]-\lambda_0,0]$; $]-\lambda_0,\lambda_0[$; or $[0,\lambda_0[$, for some $\lambda_0>0$.

Since we study local bifurcations, we use germs to define branches. 
Let $D_1,D_2$ be domains. We say that two smooth functions $b_1:D_1\rightarrow\mathbb{R}^{|N|}$ and $b_2:D_2\rightarrow\mathbb{R}^{|N|}$ are \emph{germ equivalents} if there exists an open neighborhood $U$ of $0$ such that $U \cap D_1 \cap D_2\neq \{0\}$ and $b_1(\lambda)=b_2(\lambda)$, for every $\lambda\in U \cap D_1 \cap D_2$. 
The previous relation is not transitive, so we consider its closer by transitivity.
 Given a smooth function $b$, we use the term \emph{germ} $b$ to refer to a representative element of the equivalence class of $b$ with respect to germ equivalence. 

%SIDE NOTE:
%Consider a smooth function $b: ]-\lambda_0,\lambda_0[\rightarrow \mathbb{R}^{|N|}$. Define $b^+$ and $b^-$ as the restriction of $b$ to $[0,\lambda_0[$ and $[0,\lambda_0[$, respectively. Then $b$ and $b^+$ are germ equivalents, $b$ and $b^-$ are germ equivalents. But $b^+$ and $b^-$ are not germ equivalents.
%END SIDE NOTE

Let $D$ be a domain. We say that a germ $b:D\rightarrow \mathbb{R}^{|N|}$ is an \emph{equilibrium branch of $f$ on $N$}, if $$f^N(b(\lambda),\lambda)=0,$$ for every $\lambda \in D$.
Since $f(0,0,\dots,0,\lambda)=0$ for every $\lambda$, we have that $x(\lambda)=(0,\dots,0)$ is an equilibrium branch of $f$ on $N$, called \emph{the trivial branch of $f$ on $N$}. 
The equilibrium branches of $f$ on $N$ different from trivial branch are called the \emph{bifurcation branches of $f$ on $N$}. We define the set of equilibrium branches of $f$ on $N$:
$$\mathcal{B}(N,f)=\{b:D\rightarrow \mathbb{R}^{|N|}: b \textrm{ is a equilibrium branch of $f$ on $N$}\}.$$

\section{Steady-state bifurcations for FFNs}

%In this section, we study steady-state bifurcations for feed-forward network.

\subsection{Steady-state bifurcations for FFNs associated with the valency}\label{sec:bifffnval}

First, we study the bifurcation problem of $f^N$ when $f\in\mathcal{V}_k(N)$.

\begin{prop}\label{prop:bifbrasumnet}
Let $N$ be a feed-forward network with layers $C_0,C_1,\dots,C_m$. Let $f\in \mathcal{V}_k(N)$. 
Then, generically, there are $2^{|C_0|}$ equilibrium branches of $f$ on $N$.
Moreover every equilibrium branch is uniquely determined by its value at the cells of the first layer $C_0$. 
\end{prop}

%In the proof, we will use the fact that the restriction of any coupled cell system to a source component is independent of every other cell, because a source component is a subnetwork. 

\begin{proof}
Let $N$ be a feed-forward network with layers $C_0,C_1,\dots,C_m$. Let $f\in \mathcal{V}_k(N)$. 
Generically, assume that $f_0\neq 0$, $\sum_{i,j=0}^k f_{ij}\neq 0$ and  $\sum_{i=0}^k f_{i \lambda}\neq 0$, where 
%$f_i$ is the partial derivative of $f(x_0,x_1,\dots,x_k,\lambda)$ at $(0,0,\dots,0,0)$ with respect to the variable $x_i$, 
$f_{ij}$ is the second order partial derivatives of $f(x_0,x_1,\dots,x_k,\lambda)$ at $(0,0,\dots,0,0)$ with respect to $x_i$ and $x_j$, 
and $f_{i \lambda}$ is the second order partial derivatives of $f$ at $(0,0,\dots,0,0)$ with respect $x_i$ and $\lambda$, 
for $0\leq i,j\leq k$.

The equilibrium branches of $f$ on $N$ are given by the solutions of $$f^N(x,\lambda)=0,$$ in a neighborhood of the origin.
The Taylor expansion of $f$ at $(0,0,\dots,0,0)$ is given by
\begin{align*}
f(x,x_1,\dots,x_k,\lambda)=& \sum_{i=0}^{k}f_i x_i + \sum_{i=0}^k f_{i\lambda}x_i \lambda  +  \sum_{i,j=0}^k  \frac{f_{ij}}{2}x_i x_j+ h.o.t.,
\end{align*}
where $h.o.t$ denotes high order terms.

For $c\in C_0$, we have that 
\begin{align*}
f^N_{c}(x,\lambda)=0\Leftrightarrow& f(x_c,x_c,\dots,x_c,\lambda)=0.\\
\Leftrightarrow & x_c \lambda\sum_{i=0}^k f_{i\lambda}  +  x_c^2\sum_{i,j=0}^k \frac{f_{ij}}{2}+ h.o.t.=0\\
\Leftrightarrow & x_c=0 \vee \lambda\sum_{i=0}^k f_{i\lambda}  +  x_c\sum_{i,j=0}^k  \frac{f_{ij}}{2}+ h.o.t.=0
\end{align*}

Using the Implicit Function Theorem, there exist $\lambda_0>0$ and a germ $\beta:]-\lambda_0,\lambda_0[\rightarrow \mathbb{R}$ such that $\beta(0)=0$ and  
$$f(x_c,x_c,\dots,x_c,\lambda)=0 \Leftrightarrow  x_c=0 \vee x_c= \beta(\lambda),\quad\quad -\lambda_0<\lambda<\lambda_0.$$

Denote by $D$ the set of cells $C_1\cup \dots\cup C_m$. Since $f_0\neq 0$, the matrix $\left[\partial f_{d}^{N}/\partial x_{d'}\right]_{d,d'\in D}$is invertible.
By the Implicit Function Theorem, there exist $\lambda'_0>0$ and $W:\mathbb{R}^{|C_0|}\times]-\lambda'_0,\lambda'_0[\rightarrow\mathbb{R}^{|D|}$ such that $\lambda'_0\leq \lambda_0$ and
\begin{align*}
f^N(x,\lambda)=0 \Leftrightarrow& \left(\bigwedge_{c\in C_0} f(x_c,x_c,\dots,x_c,\lambda)=0\right) \wedge x_{D}=W(x_{C_0},\lambda).\\
 \Leftrightarrow&  \left(\bigwedge_{c\in C_0} [x_{c}=0 \vee x_c= \beta(\lambda)]\right) \wedge x_{D}=W(x_{C_0},\lambda).
\end{align*}
for $-\lambda'_0<\lambda<\lambda'_0$.

Therefore any equilibrium branch is uniquely determined by its value at the cells of the first layer $C_0$
and each cell of $C_0$ has one of two possible values. 
So there are $2^{|C_0|}$ equilibrium branches.
\end{proof}

\subsection{Steady-state bifurcations for  FFNs associated with the internal dynamics}\label{sec:ffnssbar}

Next, we focus on the bifurcation problem when $f\in \mathcal{V}_0(N)$.
We follow the analysis done in \cite[Section 2]{RS13}.

\begin{defi}[{\cite[Definition 2.2]{RS13}}]\label{def:RS1322}
Let $N$ be a feed-forward network and $b:D\rightarrow \mathbb{R}^{|N|}$ a bifurcation branch on $N$, where $D=[0,\lambda_0[$ or $D=]-\lambda_0,0]$.
For any cell $c$ in $N$ such that $b_c\neq 0$, we say that $b_c$ \emph{has square-root-order $p_c$ and slope $s_c$} and write that  $b_c\sim \mathcal{O}( 2^{ -p_c})$, if $p_c$ is the smallest non-negative integer such that there is a smooth function $b^*_c:[0,\lambda_0^{2^{-p_c}}[\rightarrow \mathbb{R}$ satisfying $$b_c(\lambda)=  b^*_c(|\lambda|^{2^{-p_c}}), \quad s_c=\lim_{\substack{|\lambda|\searrow 0\\ \lambda\in D}}\frac{b^*_c(|\lambda|)}{\lambda}\neq 0.$$
If $b_c=0$, we say that $b_c$ \emph{has square-root-order $-1$ and slope $0$}.
For a subset of cells $A$, we say that $b$ \emph{has square-root-order $p$ in $A$} and we write that $b_A\sim \mathcal{O}(2^{ -p})$, if  \[p=\underset{c\in A}{\max}\left\{p_c: b_c\sim \mathcal{O}(2^{ -p_c})\right\}.\qedhere\]
\end{defi}

In \cite{RS13}, the authors describe the bifurcation branches on feed-forward networks that have only one cell in each layer, \cite[Theorem 2.3]{RS13}. Using the same ideas, we study the bifurcation branches on any feed-forward network. In the next two lemmas, we see how the square-root-order of a solution to $f(x,x_1,\dots,x_k,\lambda)=0$ grows when the square-root-order of the inputs $x_1,\dots,x_k$ are known. In the first lemma, we consider inputs that are defined for positive values of the parameter $\lambda$.

\begin{lem}\label{lem:growsropos}
Let $f\in\mathcal{V}_0(N)$ generic, $y:[0,\lambda_0[\rightarrow \mathbb{R}^k$ a germ, $p_1,\dots,p_k$ and $s_1,\dots,s_k$ such that $y_i$ has square-root-order $p_i$ and slope $s_i$ for $1\leq i\leq k$. Suppose that $p:=\max\{p_1,\dots,p_k\}\geq 0$ and define $$A:=\left\{i: y_i \sim \mathcal{O}(2^{-p})\right\}
%,\quad \quad a_i:=\lim_{\substack{\displaystyle|\lambda|\rightarrow 0\\ \displaystyle\lambda\in D}} \frac{y_i(\lambda)}{|\lambda|^{\displaystyle 2^{ -p}}}
,\quad\quad Z=\sum_{i\in A}\frac{f_i s_i}{f_{00}}.$$

\noindent
$(i)$ If $Z<0$, then there exist $0<\lambda^*_0<\lambda_0$ and germs $b^{+}, b^{-}:[0,\lambda^*_0[\rightarrow \mathbb{R}$ such that $b^{\pm}$ have square-root-order $p+1$ and slope $\pm \sqrt{-2Z}$, and
$$f(x,y(\lambda),\lambda)=0 \Leftrightarrow x=b^{\pm}(\lambda),\quad \quad 0<\lambda< \lambda^*_0$$

\noindent
$(ii)$ If $Z>0$, then the equation $f(x,y(\lambda),\lambda)=0$ has only the trivial solution $(x,\lambda)=(0,0)$.
\end{lem}

\begin{proof}
Let $f\in\mathcal{V}_0(N)$, $y:[0,\lambda_0[\rightarrow \mathbb{R}^k$, $p_1,\dots,p_k$ and $s_1,\dots,s_k$ such that $y_i$ has square-root-order $p_i$ and slope $s_i$ for $1\leq i\leq k$. Suppose that $p:=\max\{p_1,\dots,p_k\}\geq 0$ and define $$A:=\left\{i: y_i \sim \mathcal{O}(2^{-p})\right\},\quad\quad Z=\sum_{i\in A}\frac{f_i s_i}{f_{00}}.$$
 
Recall the Taylor expansion of $f$ at the origin
\begin{align*}
f(x,x_1,\dots,x_k,\lambda)=& \sum_{i=1}^{k}f_i x_i+ \frac{f_{00}}{2}x^2 +f_{0\lambda}x\lambda+ \sum_{i=1}^k f_{i\lambda}x_i \lambda  +\\
+& \sum_{i=1}^k f_{0i}x_i x + \sum_{i,j=1}^k \frac{f_{ij}}{2}x_i x_j+ h.o.t..
\end{align*}

For $\lambda\geq 0$, consider the following transformation of variables 
$$\mu=\lambda^{ 2^{ -(p+1)}},\quad\quad x=\mu z,\quad\quad y_i(\lambda)=\lambda^{2^{-p_i}}w_i(\mu).$$
Then 
$$w_i(0)=\lim_{\lambda\searrow 0} \frac{y_i(\lambda)}{\lambda^{2^{-p_i}}}
%=\lim_{\lambda\searrow 0} \frac{y^*_i(|\lambda|^{2^{-p_i}})}{\lambda^{2^{-p_i}}}= \lim_{\lambda\searrow 0} \frac{y^*_i(|\lambda|)}{\lambda}
=s_i,\quad\quad\lambda=\mu^{2^{(p+1)}},\quad\quad y_i(\lambda)=\mu^{ 2^{(p+1-p_i)}}w_i(\mu).$$
Moreover $p-p_i= 0$, if  $i\in A$, and $p-p_i> 0$, otherwise.
Using the transformation of variables and the Taylor expansion of $f$, we obtain that
\begin{align*}
f(x,y(\lambda),\lambda)=0\Leftrightarrow &
\sum_{i=1}^{k}f_i y_i(\lambda)+ \frac{f_{00}}{2}x^2 +f_{0\lambda}x\lambda+ \sum_{i=1}^k f_{i\lambda}y_i(\lambda) \lambda  +\\
& + \sum_{i=1}^k f_{0i}y_i(\lambda) x + \sum_{i,j=1}^k \frac{f_{ij}}{2}y_i(\lambda) y_j(\lambda)+ h.o.t.=0\\
%\sum_{i=1}^{k}f_i y_i(\lambda)+ \frac{f_{00}}{2}x^2 +f_{0\lambda}x\lambda + h.o.t.
%\\ + \mathcal{O}\left(|x|^3+|x|^2|\lambda|+|x||\lambda|^2+\sum_{i=1}^k |y_i(\lambda)|(|\lambda|+|x|+\sum_{j=1}^k |y_j(\lambda)|)\right)
%=0\\
\Leftrightarrow&
\sum_{i\in A}f_i \mu^{2}w_i(0)+ \frac{f_{00}}{2}\mu^2 z^2 + h.o.t.
%\mathcal{O}\left(|\mu|^2(|\mu| |z|+|\mu|^{2})\right)
=0\\
\Leftrightarrow&\mu^{2}\left(\sum_{i\in A}f_i w_i(0)+ \frac{f_{00}}{2} z^2 + h.o.t.
% \mathcal{O}(|\mu| |z|+|\mu|^{2})
\right)=0\\
\Leftrightarrow&\mu=0\vee \sum_{i\in A}f_i s_i+ \frac{f_{00}}{2} z^2 + h.o.t.
%\mathcal{O}\left(|\mu| |z|+|\mu|^{2}\right)
=0.
\end{align*}

Define $$h(z,\mu)=\sum_{i\in A}f_i s_i+ f_{00} z^2/2 + h.o.t..$$ 

If $Z<0$, we have that $h(\pm\sqrt{-2 Z},0)=0$ and $h_z(\pm\sqrt{-2 Z},0)\neq 0$. By the Implicit Function Theorem, there exist a neighborhood $U$ of $0$ and functions $z^{+}, z^{-}:U\rightarrow \mathbb{R}$ such that 
$$h(z,\mu)=0\Leftrightarrow z=z^{\pm}(\mu), \quad \quad z^{\pm}(\mu)=\pm \sqrt{-2 Z}+h.o.t.$$ 

Let $0<\lambda^*_0<\lambda_0$ and $b^{+},b^{-}:[0,\lambda^*_0[\rightarrow \mathbb{R}$ such that $[0,(\lambda^*_0)^{2^{-(p+1)}}[\subseteq U$ and 
$$b^{\pm}(\lambda)=\mu z^{\pm}(\mu)=  \pm \sqrt{-2 Z} \lambda^{ 2^{-(p+1)}}+h.o.t.\sim \mathcal{O}(2^{-(p+1)}).$$
 Then $b^{\pm}$ have square-root-order $p+1$ and slope $\pm \sqrt{-2Z}$, and
$$f(x,y(\lambda),\lambda)=0\Leftrightarrow \mu=0\vee  z= z^{\pm}(\mu)\Leftrightarrow  \mu z= \mu z^{\pm}(\mu)\Leftrightarrow  x= b^{\pm}(\lambda).
$$
This proves $(i)$.

If $Z>0$, then $h(z,0)$ is always positive, when $f_{00}>0$, or it is always negative, when $f_{00}<0$. So there is no solution to the equation $h(z,0)=0$. And the equation $f(x,y(\lambda),\lambda)=0$ has only the trivial solution $(x,\lambda)=(0,0)$, proving $(ii)$.
\end{proof}

The second lemma consider inputs defined for negative values of the parameter $\lambda$. Since the proof is very similar to the previous one, we omit it.

\begin{lem}\label{lem:growsroneg}
Let $f\in\mathcal{V}_0(N)$ generic, $y:]-\lambda_0,0]\rightarrow \mathbb{R}^k$ a germ, $p_1,\dots,p_k$ and $s_1,\dots,s_k$ such that $y_i$ has square-root-order $p_i$ and slope $s_i$ for $1\leq i\leq k$. Suppose that $p:=\max\{p_1,\dots,p_k\}\geq 0$ and define $$A:=\left\{i: y_i \sim \mathcal{O}(2^{-p})\right\},\quad\quad Z=\sum_{i\in A}\frac{f_i s_i}{f_{00}}.$$

\noindent
$(i)$ If $Z>0$, then there exist $\lambda_0<\lambda^*_0<0$ and germs $b^{+}, b^{-}:]-\lambda^*_0,0]\rightarrow \mathbb{R}$ such that $b^{\pm}$ have square-root-order $p+1$ and slope $\mp \sqrt{2Z}$, and
$$f(x,y(\lambda),\lambda)=0 \Leftrightarrow x=b^{\pm}(\lambda),\quad\quad 0>\lambda> \lambda^*_0.$$

\noindent
$(ii)$ If $Z<0$, then the equation $f(x,y(\lambda),\lambda)=0$ has only the trivial solution $(x,\lambda)=(0,0)$.
\end{lem}

Now, we describe the square-root-orders of any bifurcation branch in the layers of a feed-forward network.

\begin{prop}\label{bifbrasquarootorderlay}
Let $N$ be a feed-forward network with layers $C_0,C_1,\dots,C_m$ and $f\in\mathcal{V}_0(N)$ generic. If $b$ is a bifurcation branch of $f$ on $N$, then
%\footnote{ $f_{00}\neq 0$, $f_{0\lambda}\neq 0$ and $\underset{i_1\in I_1}{\sum}(\pm f_{i_1})\sqrt{\left|\underset{i_2\in I^{i_1}_2}{\sum}f_{i_2}
%																										\sqrt{\left|\dots 
%																												\sqrt{\left|f_{0\lambda}\underset{i_{j}\in I^{i_1^{\cdot^{\cdot^{i_{j-1}}}}}_{j}}{\sum} f_{i_{j}}\right|}
%																										\right|}
%																								\right|}\neq 0$ for every $1\leq j\leq m-1$ and subsets $I_1,I^{i_1}_2,\dots,I^{i_1^{\cdot^{\cdot^{i_{j-1}}}}}_{j}\subseteq\{1,\dots,k\}$, where $i_{l}\in I_l^{i_1^{\cdot^{\cdot^{i_{l-1}}}}}$ and $1\leq l \leq j-1$}
there exists $1\leq r\leq m$ such that $$b_{C_0}=\dots= b_{C_{r-1}}=0, b_{C_{r}}\sim \mathcal{O}(2^{ 0}), b_{C_{r+1}}\sim \mathcal{O}(2^{ -1}), \dots, b_{C_m}\sim \mathcal{O}(2^{ (r-m)}).$$
\end{prop}

\begin{proof}
Let $N$ be a feed-forward network with layers $C_0,C_1,\dots,C_m$, and $f\in\mathcal{V}_0(N)$. Denote by $(\sigma_i)_{i=1}^k$ the functions that represent $N$. Let $b$ be a bifurcation branch of $f$ on $N$. Since $b\neq 0$, define $$r=\min\{j: b_{C_j}\neq 0\}.$$ 

We check first that $1\leq r\leq m$. 
The restriction of $f^{N}(x,\lambda)=0$ to $C_0$ is equivalent to $f(x_c,x_c,\dots,x_c,\lambda)=0$, for every $c\in C_0$.  
Generically, we assume that $\sum_{i=0}^{k} f_j\neq 0$. 
By the Implicit Function Theorem, there exists an open neighborhood $D$ of $0$ such that $f(x_c,x_c,\dots,x_c,\lambda)=0$ if and only if  $x_c(\lambda)=0$, for $\lambda\in D$. Hence $b_{C_0}(\lambda)=(0,\dots,0)$ and $1\leq r\leq m$.

Now, we prove the result when $m=r$. Then $b_{C_0\cup C_1\cup \dots\cup C_{m-1}}=0$ and $f(b_c,0,\dots,0,\lambda)=0$, for every $c\in C_m$.
Generically, we can assume that $f_{00}\neq 0$ and $f_{0\lambda}\neq 0$.
The Taylor expansion of $f(x,0,\dots,0,\lambda)$ at the origin is
\begin{align*}
f(x,0,\dots,0,\lambda)=& f_{0\lambda}x\lambda+\frac{f_{00}}{2}x^2+h.o.t.\\
=& x(f_{0\lambda}\lambda+\frac{f_{00}}{2}x+h.o.t.)\\
=& x h(x,\lambda).
\end{align*}
Hence $$f(x,0,\dots,0,\lambda)=0\Leftrightarrow x=0 \vee h(x,\lambda)=0.$$ Note that $h(0,0)=0$ and $h_x(0,0)=f_{00}/2\neq 0$. Then there exist an open neighborhood $D$ of $0$ and a germ $b^0:D\rightarrow \mathbb{R}$ such that 
\begin{equation}\label{eq:germwithsro0}
b^0(\lambda)= -2\frac{f_{0\lambda}}{f_{00}}\lambda+h.o.t.,%\sim \mathcal{O}(2^0),
\end{equation}
and 
$$h(x,\lambda)=0\Leftrightarrow x=b^0(\lambda),\quad\quad \lambda\in D.$$
 Moreover, $f(x,0,\dots,0,\lambda)=0$ if and only if $x(\lambda)=0$ or $x(\lambda)=b^0(\lambda)$. Then $b_{C_m}\in\{0,b^0\}^{|C_m|}$. 
Note that the restrictions of $b^0$ to the positive and negative values of $\lambda$ have the same square-root-order, $0$, and slope, $-2f_{0\lambda}/f_{00}$.
Because $b\neq 0$, there exists $c\in C_m$ such that $b_c=b^0$ has square-root-order $0$. Therefore
$$b_{C_0}=b_{C_1}=\dots= b_{C_{m-1}}=0, b_{C_{m}}\sim \mathcal{O}(2^{0}).$$

In order to prove the result we use an inductive argument on the number of layers. The case $m=1$ is covered in the analysis of $r=m$.  So we assume that the result is valid for networks with $m'+1\geq 2$ layers, and we prove it for networks with $m+1=m'+2$ layers.
 If $r=m$, then the result holds. 
Suppose that $r\leq m'$. 
Let $N'$ be the restriction of $N$ to the first $m'+1$ layers of $N$. 
Then $b'=b_{C_0\cup\dots\cup C_{m'}}$ is an equilibrium branch of $f$ on $N'$. 
From $r\leq m'$, we know that $b'\neq 0$ and $b'$ is a bifurcation branch of $f$ on $N'$. 
Since $r=\min\{j: b'_{C_j}\neq 0\}$, we have, by induction hypothesis, that 
$$ b_{C_0}=b_{C_1}=\dots= b_{C_{r-1}}=0, b_{C_{r}}\sim \mathcal{O}(2^{ 0}), \dots, b_{C_{m'}}\sim \mathcal{O}(2^{ (r-m')}).$$

For every $c\in C_m$ and $1\leq i\leq k$, we have that $\sigma_i(c)\in C_{m-1}$, $b_{\sigma_i(c)}$ has square-root-order $p_i\leq m-r-1$ and $b_c$ is a solution of the system $$f(x_c,b_{\sigma_1(c)}(\lambda),\dots,b_{\sigma_k(c)}(\lambda),\lambda)=0.$$
Following Lemmas~\ref{lem:growsropos} and ~\ref{lem:growsroneg}, define $p_c=\max\{p_1,\dots,p_k\}\leq m-r-1$.
If $p_c\geq 0$, then $b_c$ has square-root-order $p_c+1\leq m-r$. 
If $p_c=-1$, then $b_{\sigma_i(c)}(\lambda)=0$ for every $1\leq i\leq k$ and $b_c$ has square-root-order $-1$ or $0$.
Hence $b_c$ has square-root-order less or equal to $m-r$, for every $c\in C_m$.

Because $b_{C_{m-1}}\sim \mathcal{O}( 2^{(r-m+1)})$, there exists $c'\in C_{m-1}$ such that $b_{c'}$ has square-root-order $m-r-1$. And there exists $i$ and a cell $c\in C_m$ such that $\sigma_i(c)=c'$. So $p_c= m-r-1$ and $b_c$ has square-root-order $p_c+1=m-r$, by Lemmas~\ref{lem:growsropos} and ~\ref{lem:growsroneg}. Hence 
\[ b_{C_{m}}\sim \mathcal{O}( 2^{(r-m)}).\qedhere\]
\end{proof}

As we saw, the definitions of square-root-order and slope naturally extend to bifurcation branches defined in a neighborhood of $0$.
%Inspired by the characterization of steady-state solution
We note that Lemmas~\ref{lem:growsropos} and ~\ref{lem:growsroneg} also provide a formula to classify every bifurcation branch on a feed-forward network.

Let $N$ be a feed-forward network and $f\in\mathcal{V}_0(N)$. We define 
\begin{align*}
\Theta:\mathcal{B}(N,f)&\rightarrow \{-1,0,1\}\times \mathbb{Z}^{|N|}\times \mathbb{R}^{|N|}\\
b&\mapsto (\delta, (p_c)_{c} , (s_c)_{c}),
\end{align*}
where $\delta$ is $0$ if some function in the germ equivalence class of $b$ is defined in an open neighborhood of $0$, $\delta$ is $1$ ($-1$) if $b$ is defined only for positive (negative, respectively) values of $\lambda$, and $b_c$ has square-root-order $p_c$ and slope $s_c$, for each cell $c$ of $N$.
If $(\delta,(p_c)_c, (s_c)_c)\in \Theta(\mathcal{B}(N,f))$, then
\begin{enumerate}[label=$\Omega$.\arabic*]
%\item $c\in C_0\Rightarrow p_c=-1$
\item $\delta=0 \Rightarrow \displaystyle \forall_c\  p_c\leq 0$,  \label{cond:ome1}
\item $p_c = -1 \Rightarrow \displaystyle \forall_i\  p_{\sigma_i(c)}=-1$,\label{cond:ome2}
\item $p_c > -1 \Rightarrow \displaystyle \forall_i\  p_{\sigma_i(c)}\leq p_c-1 \wedge \exists_{i'}\  p_{\sigma_{i'}(c)}= p_c-1$,\label{cond:ome3}
\item $p_c = -1 \Leftrightarrow s_c=0$,\label{cond:ome4}
\item $p_c =  0 \Rightarrow \displaystyle s_c=-\dfrac{2f_{0\lambda}}{f_{00}}$,\label{cond:ome5}
\item $p_c >  0 \Rightarrow s_c=\pm\sqrt{\displaystyle -\frac{2\delta}{f_{00}}\sum_{i\in A_c} f_i s_{\sigma_i(c)}}$,\label{cond:ome6}
\end{enumerate}	
where $A_c=\{i:p_{\sigma_i(c)}=p_c-1\}$. 
%The first property follows from Proposition~\ref{bifbrasquarootorderlay}, since $b_{C_0}=0$.
The statements \ref{cond:ome1}, \ref{cond:ome2} and \ref{cond:ome3} follow from Lemmas~\ref{lem:growsropos} and ~\ref{lem:growsroneg}, by reduction to the absurd. 
The equivalence \ref{cond:ome4} follows from  Definition~\ref{def:RS1322}. 
The statement \ref{cond:ome5} follows from the proof of Proposition~\ref{bifbrasquarootorderlay} and \ref{cond:ome3}, since $p_{\sigma_i(c)}=-1$ for every $1\leq i\leq k$. 
Finally, \ref{cond:ome6} follows from Lemma~\ref{lem:growsropos}, if $\delta=1$, and Lemma~\ref{lem:growsroneg}, if $\delta=-1$. 

Let $\Omega(N,f)\subseteq \{-1,0,1\}\times \mathbb{Z}^{|N|}\times \mathbb{R}^{|N|}$ be the set of points $(\delta,(p_c)_c, (s_c)_c)\in \{-1,0,1\}\times \mathbb{Z}^{|N|}\times \mathbb{R}^{|N|}$ satisfying  \ref{cond:ome1},\dots, \ref{cond:ome6}.
Next, we prove that $\Theta$ is an one-to-one correspondence between $\mathcal{B}(N,f)$ and $\Omega(N,f)$.

\begin{prop}\label{prop:equivtheta}
Let $N$ be a feed-forward network and $f\in\mathcal{V}_0(N)$ generic. If $(\delta,(p_c)_c, (s_c)_c)\in\Omega(N,f)$, then there exists a unique $b\in\mathcal{B}(N,f)$
such that $$\Theta(b)=(\delta,(p_c)_c, (s_c)_c).$$
\end{prop}

\begin{proof}
Let $N$ be a feed-forward network with layers $C_0,C_1,\dots, C_m$ and represented by the function $(\sigma_i)_{i=1}^k$ and $f\in\mathcal{V}_0(N)$ generic. 

Let $(\delta,(p_c)_c, (s_c)_c)\in\Omega(N,f)$. 
We construct the equilibrium branch $b$ of $f$ on $N$ such that  $\Theta(b)=(\delta,(p_c)_c, (s_c)_c)$. 
It follows from \ref{cond:ome3} that $p_c=-1$ for every $c\in C_0$ and $-1\leq p_c\leq m-1$, for every cell $c$ of $N$.

Let $c$ be a cell of $N$ such that $p_c=-1$. Then $s_c=0$, by \ref{cond:ome4}. Define $b_c$ as the germ defined on an open neighborhood of $0$ such that $b_c=0$. Then $b_c$ has square-root-order $p_c$ and slope $s_c$. It follows from \ref{cond:ome2} that 
$$f(b_c,b_{\sigma_1(c)},\dots,b_{\sigma_k(c)},\lambda)=f(0,0,\dots,0,\lambda)=0.$$

Let $c$ be a cell of $N$ such that $p_c=0$. Then $s_c= -2f_{0\lambda}/f_{00}$, by \ref{cond:ome5}, and $p_{\sigma_i(c)}=-1$ for $1\leq i\leq k$, by \ref{cond:ome3}. Define $b_c$ as the germ $b^0$ defined in (\ref{eq:germwithsro0}) on an open neighborhood of $0$. Then $b_c$ has square-root-order $p_c$, slope $s_c$ and
$$f(b_c,b_{\sigma_1(c)},\dots,b_{\sigma_k(c)},\lambda)=f(b^0,0,\dots,0,\lambda)=0.$$

The following germs are defined by induction on $p\geq 1$, i.e. we assume, for every cell $c'$ of $N$ such that $p_{c'}<p$, that $b_{c'}$ is germ which has square-root-order $p_{c'}$ and slope $s_{c'}$ and we define, for every cell $c$ of $N$ such that $p_{c}=p$, the germ $b_c$ which has square-root-order $p_{c}$ and slope $s_{c}$. Since  $p_c\leq m-1$, this process must terminate.
Let $p\geq 1$. Assume that $b_{c'}$ is a germ which has square-root-order $p_{c'}$ and slope $s_{c'}$, for every cell $c'$ such that $p_{c'}<p$.
Let $c$ be a cell of $N$ such that $p_c=p$.
Then $b_{\sigma_i(c)}$ is defined for every $1\leq i\leq k$, by \ref{cond:ome3}.
Consider the germ $y:D\rightarrow \mathbb{R}^{k}$ such that $y_i=b_{\sigma_i(c)}$ for every $1\leq i\leq k$, and let $b_c$ be the germ obtained in Lemma~\ref{lem:growsropos} (\ref{lem:growsroneg}), if $\delta=1$ ($-1$, respectively), such that $b_c$ has square-root-order $p_{c}$ slope $s_c$ and it is defined for positive (negative) values. It follows from \ref{cond:ome6} that there exists such germ and it is unique. Moreover,
$$f(b_c,b_{\sigma_1(c)},\dots,b_{\sigma_k(c)},\lambda)=0.$$

Define the germ $b=(b_c)_c:D\rightarrow \mathbb{R}^{|N|}$, where $D$ is the intersection of the domains of each $b_c$. By construction
$f^N(b(\lambda),\lambda)=0$, so $b$ is an equilibrium branch of $f$ on $N$.  Let $(\delta',(p'_c)_c, (s'_c)_c):=\Theta(b)$.
By construction, $p'_c=p_c$ and $s'_c=s_c$, for every cell $c$. 
If $\delta=0$, then $p_c\leq 0$ and $\delta'=0$, by \ref{cond:ome1}.
If $\delta=\pm 1$, then there exists $p_c>0$, by \ref{cond:ome4} and \ref{cond:ome6}. If $\delta=1$, then $b_c$ is defined for positive values and $\delta'=1$. Similar, if $\delta=-1$, then $\delta'=1$. Therefore 
$$\Theta(b)=(\delta,(p_c)_c, (s_c)_c).$$

We can see that in each step of the construction of $b$ that we choose the unique germ that respects the conditions of square-root-order, slope and be a solution to the equation.
\end{proof}

Let $N$ be a feed-forward network with layers $C_0,C_1,\dots,C_m$ and $f:\mathbb{R}^{k+1}\times\mathbb{R}\rightarrow \mathbb{R}\in\mathcal{V}_0(N)$ generic. Define 
\begin{equation}\label{eq:delta0tilde}
\tilde{\delta}=\sign(\displaystyle f_{0\lambda}\sum_{i=1}^k f_i)=\dfrac{\displaystyle f_{0\lambda}\sum_{i=1}^k f_i}{\left|\displaystyle f_{0\lambda}\sum_{i=1}^k f_i\right|},\quad \tilde{p}_0=-1,\quad \tilde{s}_0=0,
\end{equation}
and
\begin{equation}\label{eq:jtilde}
\tilde{p}_j=j-1,\quad \tilde{s}_j=- \sign(f_{0\lambda})\dfrac{2|f_{0\lambda}|^{2^{-(j-1)}}}{f_{00}}\left|\sum_{i=1}^k f_i \right|^{1-2^{-(j-1)}},
\end{equation}
for $1\leq j\leq m$.
Now, for each $1\leq r\leq m-1$, define $\delta^{r\pm}=\tilde{\delta}$,
$$p^{r\pm}_c=\tilde{p}_0,\quad s^{r\pm}_c=\tilde{s}_0, \quad\quad c\in C_0\cup\dots\cup C_{r-1},$$
$$ p^{r\pm}_c= \tilde{p}_{l+1},\quad s^{r\pm}_c=\tilde{s}_{l+1} \quad \quad  c\in C_{r+l}, 0<l<m-1-r,$$
$$ p^{r\pm}_c= \tilde{p}_{m-r},\quad s^{r\pm}_c=\pm\tilde{s}_{m-r} \quad \quad  c\in C_{m},$$
We also define $\delta^m= 0$,
$$p^m_c=-1, s^m_c=0, \quad c\in C_0\cup\dots\cup C_{m-1}, \quad p^{m}_c= 0, s^{m}_c=-\dfrac{2f_{0\lambda}}{f_{00}}, \quad c\in C_{m},$$
%and
$$\delta^0= 0,\quad\quad p^0_c=-1, \quad\quad s^0_c=0, \quad c\in C_0\cup\dots\cup C_{m}.$$
For every $1\leq r\leq m-1$, $$(\delta^0,(p^0_c)_c,(s^0_c)_c), (\delta^m,(p^m_c)_c,(s^m_c)_c), (\delta^{r\pm},(p^{r\pm}_c)_c,(s^{r\pm}_c)_c)\in \Theta(N,f).$$

By Proposition~\ref{prop:equivtheta}, the set $\mathcal{B}(N,f)$ contains the trivial equilibrium branch $b^0$, 
a bifurcation branch $b^{m}$ such that 
$$ b^{m}_{C_0}=\dots= b^{m}_{C_{m}}=0, b^{m}_{C_{m}}\sim \mathcal{O}( 2^{0}),$$
and for every $1\leq r\leq m-1$ there exist two bifurcation branches $b^{r+}, b^{r-}$ such that 
$$ b^{r\pm}_{C_0}=\dots= b^{r\pm}_{C_{r-1}}=0, b^{r\pm}_{C_{r}}\sim \mathcal{O}( 2^{ 0}), b^{r\pm}_{C_{r+1}}\sim \mathcal{O}(2^{ -1}), \dots, b^{r\pm}_{C_m}\sim \mathcal{O}(2^{(r-m)}).$$
Hence there exists a bifurcation branch with square-root-order $r$, for every  $r$.

\begin{coro}\label{coro:existbifbrawithorder}
Let $N$ be a feed-forward network and $f\in\mathcal{V}_0(N)$. Generically, for every $1\leq r\leq m$, there exists $b\in\mathcal{B}(N,f)$ such that $$ b_{C_0}=\dots= b_{C_{r-1}}=0, b_{C_{r}}\sim \mathcal{O}( 2^{ 0}), b_{C_{r+1}}\sim \mathcal{O}(2^{ -1}), \dots, b_{C_m}\sim \mathcal{O}(2^{(r-m)}).$$
\end{coro}

In \cite[Theorem 2.3]{RS13}, the authors prove that the germs $b^0=(b^0_c)_c$, $b^{r\pm}=(b^{r\pm}_c)_c$ and $b^{m}=(b^{m}_c)_c$, where $1\leq r\leq m-1$ are the unique equilibrium branches of $f$ on $N$, if $N$ has only one cell in each layer. By examining the set $\Theta(N,f)$, we can see that $(\delta^0,(p^0_c)_c,(s^0_c)_c)$, $(\delta^m,(p^m_c)_c,(s^m_c)_c)$, $(\delta^{r\pm},(p^{r\pm}_c)_c,(s^{r\pm}_c)_c)$ are the unique elements of $\Theta(N,f)$, when $N$ has only one cell in each layer and recover \cite[Theorem 2.3]{RS13}.
The characterization of bifurcation branches is illustrated in the following example.

\begin{exe}\label{ex:sufcondisnotneccond}
Let $N$ be the feed-forward network represented in Figure~\ref{fig:sufcondisnotneccond} and $f\in\mathcal{V}_0(N)$ generic.
We assume that $f_{0\lambda}f_{00}>0$, the other case is identical.
The possible bifurcation branches of $f$ on $N$ are described in Table~\ref{tab:sufcondisnotneccond}, where 
$\delta_1=-\sign(f_{0\lambda}f_1)$, $\delta_2=-\sign(f_{0\lambda}f_2)$, $\delta_3=-\sign(f_{0\lambda}f_3)$,
$\delta_4=-\sign(f_{0\lambda}(f_1+f_2))$, $\delta_5=-\sign(f_{0\lambda}(f_1 + f_3))$, $\delta_6=-\sign(f_{0\lambda}(f_2+f_3))$,
$\delta_7=-\sign(f_{0\lambda}(f_1+f_2+f_3))$,
$\gamma_a=\pm 1$, $\kappa_1=\delta_3\delta_{7}$, $\kappa_2= \delta_1\delta_{3}$, $\kappa_3=\delta_1\delta_{7}$, $\tilde{s}^0=-2f_{0\lambda}/f_{00}$, 
$s^{1}_{a}$, $s^{2}_{a}$ and $s^{3}_{a}$ are inductively calculated using (\ref{cond:ome6}), e.g. for the bifurcation branches in row $17$,
$$\tilde{s}^1_{7}=\dfrac{2}{|f_{00}|} \sqrt{\displaystyle |(f_1+f_2+f_3) f_{0\lambda}|},$$ 
$$\tilde{s}^2_4= \dfrac{2}{|f_{00}|}
\sqrt{\displaystyle \left|\gamma_1 f_1+\gamma_2 f_2\right| \sqrt{|f_1+f_2+f_3||f_{0\lambda}|}},$$
under the following condition
\begin{equation}\label{eq:2cond123123123}
(f_1+f_2+f_3)(\gamma_1 f_1+\gamma_2 f_2)>0.
\end{equation}
The other conditions are
\begin{equation}\label{eq:2cond1112}
f_1(\gamma_1 f_1 \sqrt{|f_1|}+ \gamma_2 f_2 \sqrt{|f_1+f_2|})>0,
\end{equation}
\begin{equation}\label{eq:2cond3233}
f_3(\gamma_1 f_1 \sqrt{|f_2+f_3|}+ \gamma_2 f_2 \sqrt{|f_3|})>0,
\end{equation}
\begin{equation}\label{eq:2cond1112123}
f_1(\gamma_1f_1 \sqrt{|f_1|}+\gamma_2f_2 \sqrt{|f_1+f_2|} + \gamma_3f_3 \sqrt{|f_1+f_2+f_3|})>0,
\end{equation}
\begin{equation}\label{eq:2cond3233123}
f_3(\gamma_1f_1 \sqrt{|f_2+f_3|}+\gamma_2f_2 \sqrt{|f_3|} + \gamma_3f_3 \sqrt{|f_1+f_2+f_3|})>0,
\end{equation}
\begin{equation}\label{eq:2cond123123123123}
(f_1+f_2+f_3) (\gamma_1f_1+\gamma_2f_2 + \gamma_3f_3 )>0,
\end{equation}
\begin{equation}\label{eq:2cond31232323}
\begin{cases}
f_3(\gamma_1f_1 \sqrt{|f_1+f_2+f_3|}+\gamma_2(f_2+f_3) \sqrt{|f_2+f_3|})>0\\
f_3(\gamma_1(f_1+f_2) \sqrt{|f_1+f_2+f_3|}+ \gamma_2f_3 \sqrt{|f_2+f_3|})>0\\
%\begin{split}
\delta_3\biggl(\gamma_3 f_1\sqrt{\displaystyle \left|\gamma_1f_1 \sqrt{|f_1+f_2+f_3|}+\gamma_2(f_2+f_3) \sqrt{|f_2+f_3|}\right|}\\
\omit\hfill $ +\gamma_4 f_2\sqrt{\displaystyle \left|\gamma_1(f_1+f_2) \sqrt{|f_1+f_2+f_3|}+ \gamma_2f_3 \sqrt{|f_2+f_3|}\right|}$\\
\omit\hfill $+\gamma_5 f_3\sqrt{\displaystyle |f_1 +f_2 + f_3|\sqrt{|f_3|}}\biggr)>0$
%\end{split}
\end{cases},
\end{equation}
\begin{equation}\label{eq:3cond2}
\begin{split}
\delta_1\biggl(\gamma_1 f_1\sqrt{\displaystyle |f_2 + f_3| \sqrt{|f_1|}}
+\gamma_2 f_2\sqrt{\displaystyle |f_3| \sqrt{|f_1|}}\\
+\gamma_3 f_3\sqrt{\displaystyle |f_1+f_2+f_3|\sqrt{|f_1+f_2|}}\biggr)>0
\end{split},
\end{equation}
\begin{equation}\label{eq:2cond27a}
\begin{cases}
(f_1+f_2+f_3) (\gamma_1(f_1+f_2) + \gamma_2f_3 )>0\\
(f_1+f_2+f_3) (\gamma_1f_1+\gamma_2(f_2 +f_3) )>0\\
%\begin{split}
\delta_7\biggl(
 \gamma_3 f_1\sqrt{\displaystyle \left|\gamma_1 f_1 + \gamma_2(f_2+f_3) \right|}
+\gamma_4 f_2\sqrt{\displaystyle \left|\gamma_1(f_1+f_2) + \gamma_2 f_3 \right|}\\
\omit\hfill $+\gamma_5 f_3\sqrt{\displaystyle \left|f_1+f_2 + f_3 \right|}\biggr)>0$
%\end{split}
\end{cases}.
\end{equation}

The conditions (\ref{eq:2cond1112}), (\ref{eq:2cond3233}), (\ref{eq:2cond123123123}), (\ref{eq:2cond1112123}), (\ref{eq:2cond3233123}), (\ref{eq:2cond123123123123}), (\ref{eq:2cond31232323}), (\ref{eq:3cond2}) and (\ref{eq:2cond27a}) are generically satisfied at least for one choice of $\gamma_i$.
In the table~\ref{tab:sufcondisnotneccond}, we indicate the domain $\delta$, the slope of the bifurcation branch at each cell, $s_c$ and under which conditions the guarantee the bifurcation branches exist. 
The square-root-order $p_c$ is inferred from the slope in the following way it is $-1$ if $s_c=0$ and $p$ if $s_c=\tilde{s}^p_a$. 
By Proposition~\ref{prop:equivtheta}, we recover the bifurcation branches of $f$ on $N$ from the table's rows. 
Many rows of the table correspond to more than one bifurcation branch. 
The bifurcation branches with square-root-order greater or equal than $1$ have two different slopes in the last cell, $10$. 
Depending on the function $f$ and the network structure, the slope on the other cells can be positive or negative, and we represent this choice using $\gamma_i$. 
On the other hand the condition \ref{cond:ome6} can force one of the signals and we use $\kappa_i$.
\end{exe}

\begin{table}[ht!]
%\center
\hskip-2.0cm
\begin{tabular}{ |c |c | c | c| c | c | c | c |c |c | c | c|}
\hline
\rowcolor{gray!30}
$\delta$ & $s_1$ & $s_2$ & $s_3$ & $s_4$ & $s_5$ & $s_6$ & $s_7$ & $s_8$ & $s_9$ & $s_{10}$ & Conditions \\ \hline

%$0$ & $0$ & $0$ & $0$ & $0$ & $0$& $0$ & $0$ & $0$ & $0$ & $0$ &  \\ \hline

$0$ & $0$ & $0$ & $0$ & $0$ & $0$& $0$ & $0$ & $0$ & $0$ & $\tilde{s}^0$ &  \\ \hline

$\delta_1$ & $0$ & $0$ & $0$ & $0$ & $0$& $0$ & $\tilde{s}^0$ & $0$ & $0$ & $\pm\tilde{s}^1_{1}$ &  \\ \hline

$\delta_2$ & $0$ & $0$ & $0$ & $0$ & $0$& $0$ & $0$ & $\tilde{s}^0$ & $0$ & $\pm\tilde{s}^1_{2}$ &  \\ \hline

$\delta_3$ & $0$ & $0$ & $0$ & $0$ & $0$& $0$ & $0$ & $0$ & $\tilde{s}^0$ & $\pm\tilde{s}^1_{3}$ &  \\ \hline

$\delta_{4}$ & $0$ & $0$ & $0$ & $0$ & $0$& $0$ & $\tilde{s}^0$ & $\tilde{s}^0$ & $0$ & $\pm\tilde{s}^1_{4}$ &  \\ \hline

$\delta_{5}$ & $0$ & $0$ & $0$ & $0$ & $0$& $0$ & $\tilde{s}^0$ & $0$ & $\tilde{s}^0$ & $\pm\tilde{s}^1_{5}$ &  \\ \hline

$\delta_{6}$ & $0$ & $0$ & $0$ & $0$ & $0$& $0$ & $0$ & $\tilde{s}^0$ & $\tilde{s}^0$ & $\pm\tilde{s}^1_{6}$ &  \\ \hline

$\delta_{7}$ & $0$ & $0$ & $0$ & $0$ & $0$& $0$ & $\tilde{s}^0$ & $\tilde{s}^0$ & $\tilde{s}^0$ & $\pm\tilde{s}^1_{7}$ &  \\ \hline

$\delta_1$ & $0$ & $0$ & $0$ & $\tilde{s}^0$ & $0$& $0$ & $\gamma_1\tilde{s}^1_{1}$ & $\gamma_2\tilde{s}^1_{4}$ & $0$ & $\pm\tilde{s}^{2}_{1}$ & $\delta_1\delta_{4}>0,$ (\ref{eq:2cond1112}) \\ \hline 

$\delta_1$ & $0$ & $0$ & $0$ & $\tilde{s}^0$ & $0$& $0$ & $\gamma_1\tilde{s}^1_{1}$ & $\gamma_2\tilde{s}^1_{4}$ & $\tilde{s}^0$ & $\pm\tilde{s}^{2}_{1}$ & $\delta_1\delta_{4}>0,$  (\ref{eq:2cond1112}) \\ \hline

$\delta_3$ & $0$ & $0$ & $0$ & $0$ & $\tilde{s}^0$ & $0$ & $\gamma_1\tilde{s}^1_{6}$ & $\gamma_2\tilde{s}^1_{3}$ & $0$          & $\pm\tilde{s}^{2}_{2}$ & $\delta_3\delta_{6}>0,$ (\ref{eq:2cond3233})\\ \hline

$\delta_3$ & $0$ & $0$ & $0$ & $0$ & $\tilde{s}^0$ & $0$ & $\gamma_1\tilde{s}^1_{6}$ & $\gamma_2\tilde{s}^1_{3}$ & $\tilde{s}^0$ & $\pm\tilde{s}^{2}_{2}$ & $\delta_3\delta_{6}>0,$ (\ref{eq:2cond3233})\\ \hline

$\delta_{7}$ & $0$ & $0$ & $0$ & $0$ & $0$ & $\tilde{s}^0$ & $0$ & $0$ & $\kappa_1\tilde{s}^1_{7}$  & $\pm\tilde{s}^{2}_{3}$ &  \\ \hline

$\delta_{7}$ & $0$ & $0$ & $0$ & $0$ & $0$ & $\tilde{s}^0$ & $\tilde{s}^0$ & $0$ & $\kappa_1\tilde{s}^1_{7}$  & $\pm\tilde{s}^{2}_{3}$ & \\ \hline

 $\delta_{7}$& $0$ & $0$ & $0$ & $0$ & $0$ & $\tilde{s}^0$ & $0$ & $\tilde{s}^0$ & $\kappa_1\tilde{s}^1_{7}$  & $\pm\tilde{s}^{2}_{3}$ &  \\ \hline

 $\delta_{7}$& $0$ & $0$ & $0$ & $0$ & $0$ & $\tilde{s}^0$ & $\tilde{s}^0$ & $\tilde{s}^0$ & $\kappa_1\tilde{s}^1_{7}$  & $\pm\tilde{s}^{2}_{3}$ &  \\ \hline

$\delta_{7}$ & $0$ & $0$ & $0$ & $\tilde{s}^0$ & $\tilde{s}^0$ & $0$ & $\gamma_1\tilde{s}^1_{7}$ & $\gamma_2\tilde{s}^1_{7}$ & $0$          & $\pm\tilde{s}^{2}_{4}$ & (\ref{eq:2cond123123123}) \\ \hline

$\delta_{7}$ & $0$ & $0$ & $0$ & $\tilde{s}^0$ & $\tilde{s}^0$ & $0$ & $\gamma_1\tilde{s}^1_{7}$ & $\gamma_2\tilde{s}^1_{7}$ & $\tilde{s}^0$ & $\pm\tilde{s}^{2}_{4}$ & (\ref{eq:2cond123123123}) \\ \hline

$\delta_{1}$ & $0$ & $0$ & $0$ & $\tilde{s}^0$ & $0$ & $\tilde{s}^0$  & $\gamma_1\tilde{s}^1_{1}$ & $\gamma_2\tilde{s}^1_{4}$ & $\gamma_3\tilde{s}^1_{7}$ & $\pm\tilde{s}^{2}_{5}$ & $\delta_1\delta_{4},\delta_{1}\delta_{7}>0,$ (\ref{eq:2cond1112123}) \\ \hline

$\delta_{3}$ & $0$ & $0$ & $0$ & $0$ & $\tilde{s}^0$  & $\tilde{s}^0$  & $\gamma_1\tilde{s}^1_{6}$ & $\gamma_2\tilde{s}^1_{3}$ & $\gamma_3\tilde{s}^1_{7}$ & $\pm\tilde{s}^{2}_{6}$ & $\delta_3\delta_{6},\delta_{3}\delta_{7}>0,$ (\ref{eq:2cond3233123}) \\ \hline

$\delta_{7}$ & $0$ & $0$ & $0$ & $\tilde{s}^0$ & $\tilde{s}^0$ & $\tilde{s}^0$  & $\gamma_1\tilde{s}^1_{7}$ & $\gamma_2\tilde{s}^1_{7}$ & $\gamma_3\tilde{s}^1_{7}$ & $\pm\tilde{s}^{2}_{7}$ & (\ref{eq:2cond123123123123}) \\ \hline

$\delta_{3}$ & $0$ & $\tilde{s}^0$ & $0$ & $\gamma_1\tilde{s}^1_{7}$ & $\gamma_2\tilde{s}^1_{6}$ & $\tilde{s}^1_{3}$  & $\gamma_3\tilde{s}^{2}_{8}$ & $\gamma_4\tilde{s}^{2}_{9}$ & $\gamma_5\tilde{s}^{2}_{10}$ & $\pm \tilde{s}^{3}_{1}$ & $\delta_{3}\delta_{6},\delta_{3}\delta_{7}>0,$ (\ref{eq:2cond31232323}) \\ \hline

$\delta_{1}$ & $0$ & $0$ & $\tilde{s}^0$  & $0$          & $\kappa_2\tilde{s}^1_{1}$ & $\kappa_3\tilde{s}^1_{4}$  & $\gamma_1\tilde{s}^{2}_{11}$ & $\gamma_2\tilde{s}^{2}_{12}$ & $\gamma_3\tilde{s}^{2}_{13}$ & $\pm \tilde{s}^{3}_{2}$ & $\delta_{1}\delta_{4},\delta_3\delta_{6}>0,$ (\ref{eq:3cond2}) \\ \hline

$\delta_{1}$ & $0$ & $0$ & $\tilde{s}^0$  & $\tilde{s}^0$ & $\kappa_2\tilde{s}^1_{1}$ & $\kappa_3\tilde{s}^1_{4}$  & $\gamma_1\tilde{s}^{2}_{11}$ & $\gamma_2\tilde{s}^{2}_{12}$ & $\gamma_3\tilde{s}^{2}_{13}$ & $\pm\tilde{s}^{3}_{2}$ & $\delta_{1}\delta_{4},\delta_3\delta_{6}>0,$ (\ref{eq:3cond2}) \\ \hline

$\delta_{7}$ & $0$ & $\tilde{s}^0$ & $\tilde{s}^0$ & $\gamma_1\tilde{s}^1_{7}$ & $\gamma_2\tilde{s}^1_{7}$ & $\tilde{s}^1_{7}$  & $\gamma_3\tilde{s}^{2}_{14}$ & $\gamma_4\tilde{s}^{2}_{15}$ & $\gamma_5\tilde{s}^{2}_{16}$ & $\pm \tilde{s}^{3}_{3}$ & (\ref{eq:2cond27a}) \\ \hline
\end{tabular}
\caption{The possible bifurcation branches on the FFN in  Figure~\ref{fig:sufcondisnotneccond} for a steady-state bifurcation associated to the internal dynamics.}
\label{tab:sufcondisnotneccond}
\end{table}

%SIDE NOTE:
%The set of bifurcation branches of $\dot{x}=F(x,\lambda)$ is, generic, finite? Ler Fields Dynamics and symmetry.
%
%For generic dynamical system, it should be finite. Proof?
%
%For equivariant dynamical systems. Consider the group of the sphere $S^1$ acting in $\mathbb{R}^2$. Then every equivariant dynamical systems is radial. If $(x,\lambda)\in F^{-1}(\{0\})$, then $(|x|e^{i\theta},\lambda)\in F^{-1}(\{0\})$, for every $\theta$. In ``Dynamics and symmetry'', Field says, (page 95),  ``We prove in Chapter 7 that $S$,$S_w$ are dense subsets of $V_0$ ($C_{\infty}$-topology)'', where 
%$S$,$S_w$ include that there are a finite number of branches.
%
%For coupled cell system, is it finite?
%END SIDE NOTE

\section{Lifting bifurcation problem on FFNs}

The bifurcation branches occurring in a quotient system are lifted to bifurcation branches occurring in a lift system. Next, we define when a bifurcation branches is lifted. 
In this section we study the lifting bifurcation problem which consists on understanding if every bifurcation branches occurring in a coupled cell system associated to a lift network are lifted from bifurcation branches occurring in the coupled cell system associated to the original network.

\begin{defi}
Let $N$ be a network and $L$ a lift of $N$. 
We say that \emph{a bifurcation branch $b$ of $f$ on $L$ is lifted from $N$}, if there exists a balanced coloring $\bowtie$ in $L$ such that $b\in \Delta_{\bowtie}$ and $N=L/\bowtie$. 
\end{defi}

In the next proposition, we recover a well-know result about the bifurcation branches being inside a flow-invariant space which contains the center subspace. 
We present the proof here for completeness.
Let $A:\mathbb{R}^{d}\rightarrow \mathbb{R}^{d}$ be a linear operator from $\mathbb{R}^{d}$ to itself and $d>0$. The \emph{center subspace} of $A$ is given by
$$\ker^*(A)=\{v\in \mathbb{R}^{d}: A^k v=0 \textrm{ for some $k$}\}.$$
We denote the orthogonal complement  with respect to the usual inner product of a subspace $B\subseteq V$ by $B^{\bot}$.

\begin{prop}
Let $F:\mathbb{R}^{d}\times \mathbb{R}\rightarrow \mathbb{R}^{d}$ be a smooth function and $K\subseteq \mathbb{R}^{d}$ such that $\ker^* (DF_{(0,0)})\subseteq K$, $F(0,0)=0$ and $F(K,\lambda)\subseteq K$ for every $\lambda\in \mathbb{R}$. Suppose that there exists a function $x:D\rightarrow \mathbb{R}^{d}$ defined in a domain $D$ such that $F(x(\lambda),\lambda)=0$ for $\lambda\in D$. Then there exists a neighborhood $U$ of $0$ such that $x(\lambda)\in K$ for every $\lambda\in U\cap D$.
\end{prop}

\begin{proof}
Let $F:\mathbb{R}^{d}\times \mathbb{R}\rightarrow \mathbb{R}^{d}$ be a smooth function and $K\subseteq \mathbb{R}^{d}$ such that $\ker^* (DF_{(0,0)})\subseteq K$, $F(0,0)=0$ and $F(K,\lambda)\subseteq K$ for every $\lambda\in \mathbb{R}$.
Note that $\mathbb{R}^{d}=K\oplus K^{\bot}$. Writing every element of $\mathbb{R}^{d}$ in its decomposition in $K$ and $K^{\bot}$, $v=y+w$, where $y\in K$ and $w\in K^{\bot}$, there are $g:K\times K^{\bot} \times \mathbb{R}\rightarrow K$ and $h:K\times K^{\bot} \times \mathbb{R}\rightarrow K^{\bot}$ such that
$$\dot{v}=F(v,\lambda) \Leftrightarrow  \begin{cases}\dot{y}=g(y,w,\lambda)\\ \dot{w}=h(y,w,\lambda)\\  \end{cases}.$$
Hence $$DF_{(0,0)}=\left[\begin{matrix} D_y g_{(0,0)} &D_{w} g_{(0,0)}\\ D_y h_{(0,0)} &D_{w} h_{(0,0)} \end{matrix}\right].$$

Observe that $h(y,0,\lambda)=0$, because $K$ is invariant. Then  $D_y h_{(0,0)}=0$ and $D_{w} h_{(0,0)}$ is invertible, since $\ker^* (DF_{(0,0)})\subseteq K$.
By the implicit function theorem, there is $W:K \times \mathbb{R}\rightarrow K^{\bot}$ such that $W(0,0)=0$ and $h(y,w,\lambda)=0$ if and only if $w=W(y,\lambda)$.

From $h(y,0,\lambda)=0$, we have that $W(y,\lambda)=0$.
Therefore \[F(y,w, \lambda)=0 \Leftrightarrow g(y,0,\lambda)=0 \wedge w=0.\]
Supposing that $x$ is a solution to $F(v,\lambda)=0$, we have that $x\in K$.
\end{proof}

It follows that a necessary condition for the existence of a bifurcation branch on a lift network not lifted from the original network is that the center subspace of the coupled cell systems associated to the original network and the lift network have different dimensions.

\begin{coro}\label{coro:eigspainv}
Let $N$ be a network, $L$ a lift of $N$ associated to the coloring $\bowtie$ and $f\in \mathcal{V}(N)$. If $\ker^*(J_f^N)$ and  $\ker^*(J_f^L)$ have the same dimension, then every bifurcation branch of $f$ in $L$ belongs to $\Delta_{\bowtie}$ and is lifted from $N$.
\end{coro}

\begin{obs}\label{rem:dimker}
Let $N$ be a feed-forward network with layers $C_0,C_1,\dots,C_m$. 

\noindent (i)\, If  $f\in\mathcal{V}_k(N)$, then the dimension of $\ker^*(J^{N}_f)$ is $|C_0|$.

\noindent (ii) If  $f\in\mathcal{V}_0(N)$, then the dimension of $\ker^*(J^{N}_f)$ is  $|C_1|+\dots+|C_m|$.
\end{obs}

%\subsection{Lifting Bifurcation Problem on FFN}

%Let $N$ be a feed-forward network and $L$ a lift of $N$ such that $L$ is a backward connected feed-forward network. 

%We start by the case where $L$ is a lift that create new layers. In this case, we see that the bifurcation branches are lifted if and only if the dimension of the center subspace of a coupled cell system associated does not increase.

\subsection{Lifting bifurcation problem on FFNs associated with the valency}\label{sec:lbpffnval}

In this section, we study the lifting bifurcation problem for feed-forward systems determined by a regular function that have a bifurcation condition associated to the valency. 
We prove that every bifurcation branch is lifted if and only if the center subspace of the feed-forward systems of the original network and lifted network have the same dimension.

\begin{prop}\label{prop:lbpval}
Let $N$ be a feed-forward network, $f\in\mathcal{V}_k(N)$ generic and $L$ a feed-forward lift of $N$ such that $L$ is backward connected.

\noindent (i)\, If $L$ is a lift that creates new layers or a lift inside a layer, except the first layer, then every bifurcation branch of $f$ on $L$ is lifted from $N$.

\noindent (ii) If $L$ is a lift inside the first layer, then there is at least one bifurcation branch of $f$ on $L$ which is not lifted from $N$.
\end{prop}

\begin{proof}
Let $N$ be a feed-forward network, $f\in\mathcal{V}_k(N)$ generic and $L$ a feed-forward lift of $N$ such that $L$ is backward connected.
Denote by $C_0$ and by $C'_0$ the first layer of $N$ and $L$, respectively.

If $L$ is a lift that creates new layers or a lift inside a layer, except the first, then  $\ker^*(J_f^N)$ and $\ker^*(J_f^N)$ have the same dimension.
Recall Remark~\ref{rem:dimker}.
By Corollary~\ref{coro:eigspainv}, every bifurcation branch of $f$ on $L$ is lifted from $N$.

Suppose that $L$ is a lift inside the first layer. 
By Remark~\ref{obs:liftinslaycompsplit}, we assume that $L$ is the split of a cell $c\in C_0$ into two cells $c_1,c_2\in C'_0$ and denote by $\bowtie$ the balanced coloring in $L$ given by $c_1\bowtie c_2$. 
By Lemma~\ref{lem:unibalcolbackcon}, $\bowtie$ is the unique balanced coloring such that $L/\bowtie=N$. 
By the proof of Proposition~\ref{prop:bifbrasumnet}, we know that there exists a bifurcation branch $b\in \mathcal{B}(L,f)$ such that $b_{c_1}\neq b_{c_2}$. 
So $b\notin \Delta_{\bowtie}$ and it is not lifted from $N$.
\end{proof}

The next example shows that if we do not impose the condition of backward connectedness when we do a lift in the first layer, then every bifurcation branch on the lift network may be lifted from the original network.

\begin{exe}
Returning to the Example~\ref{exe:uniqbalcolback}, let $N$ be the feed-forward network on the left of Figure~\ref{fig:uniqbalcolback} and $L$ the feed-forward network on the right of Figure~\ref{fig:uniqbalcolback} and  $f\in\mathcal{V}_k(N)$ generic. Note that $L$ is a lift inside the first layer of $N$. In Example~\ref{exe:uniqbalcolback}, we saw that there are three balanced colorings in $L$ $\bowtie_1$ given by $1\bowtie_1 2$, $\bowtie_2$ given by $2\bowtie_2 3$ and $\bowtie_3$ given by $1\bowtie_3 3$ such that $L/\bowtie_1=L/\bowtie_2=L/\bowtie_3=N$.  
Consider a bifurcation branch $b\in \mathcal{B}(L,f)$. It follows from the proof of Proposition~\ref{prop:bifbrasumnet} that $b_1=b_2$ and $b\in \Delta_{\bowtie_1}$, $b_2=b_3$ and $b\in \Delta_{\bowtie_2}$ or $b_1=b_3$ and $b\in \Delta_{\bowtie_3}$. Therefore $b$ is lifted from $N$.
\end{exe}

\subsection{Lifting bifurcation problem on FFNs associated with the internal dynamics}\label{sec:lbfffnintdy}

In this section, we study the lifting bifurcation problem for feed-forward systems determined by a regular function that has a bifurcation condition associated to the internal dynamics. 
%We prove that every bifurcation branch is lifted if and only if the center subspace of the Jacobian matrix does not increase in dimension for lifts that create new layers, lifts inside the first layer and lifts inside the last but one layer. 
%For the lifts inside any other layer, there may be bifurcation branches that are or not lifted from the quotient network. 
%We give sufficient conditions for both cases and those conditions depend on the regular functions. 
%Therefore, it follows that those results depend both on the network structure and on the regular function defining the feed-forward system.
We start by the cases that do not depend on the regular function.

\begin{prop}\label{prop:lbpintdynfirstnew}
Let $N$ be a feed-forward network, $f\in\mathcal{V}_0(N)$ generic and $L$ a feed-forward lift of $N$.

\noindent (i) \, If $L$ is a lift inside the first layer, then every bifurcation branch of $f$ on $L$ is lifted from $N$.

\noindent (ii) If $L$ is a lift that creates new layers, then there is a bifurcation branch of $f$ on $L$ which is not lifted from $N$.
\end{prop}

\begin{proof}
Let $N$ be a feed-forward network, $f\in\mathcal{V}_0(N)$ generic and $L$ a feed-forward lift of $N$.

Suppose that $L$ is a lift inside the first layer. Then the center subspace of $J_f^N$ and  $J_f^L$ have the same dimension. By Corollary~\ref{coro:eigspainv}, every bifurcation branch of $f$ on $L$ is lifted from $N$.

Suppose that $L$ is a lift that creates new layers. 
By Corollary~\ref{coro:existbifbrawithorder} and Proposition~\ref{bifbrasquarootorderlay}, there exists a bifurcation branch $b\in\mathcal{B}(L,f)$ having square-root-order greater than any bifurcation branch of $f$ on $N$. 
Hence there is a bifurcation branch of $f$ on $L$ which is not lifted from $N$.
\end{proof}

%Next, we study the lifting bifurcation problem for lifts inside a layer, except the first. Despite the fact that the center subspace of the Jacobian matrix increases, the bifurcation branches can be or not lifted from the original network depending on the functions that defines the feed-forward system.
%The next results give sufficient conditions for both cases.
%First, the cases where the bifurcation branches are not lifted.

For a lift inside a layer, except the first, such that the next layer has only one cell, there is a bifurcation branch on the lift network not lifted from the original network. In particular, this happens for backward connected lifts inside the last but one layer.

\begin{prop}\label{prop:lbpintdynextonecell}
Let $N$ be a feed-forward network  with layers $C_0,\dots,C_m$, $f\in\mathcal{V}_0(N)$ generic and $L$ a feed-forward lift of $N$.

If $L$ is a lift inside $C_j$, $0< j<m$ and $|C_{j+1}|=1$, then there is a bifurcation branch of $f$ on $L$ which is not lifted from $N$.
\end{prop}

\begin{proof}
Let $N$ be a feed-forward network  with layers $C_0,\dots,C_m$, $f\in\mathcal{V}_0(N)$ generic and $L$ a feed-forward lift of $N$.
Suppose that $L$ is a lift inside $C_j$, $0< j<m$ and $C_{j+1}=\{d\}$.
Denote by $C'_{j}$ the $(j+1)$-layer of $L$ and by $(\sigma_i^L)_{i=1}^k$ the representative functions of $L$. 
By Remark~\ref{obs:liftinslaycompsplit}, we assume that $L$ is the split of a cell $c\in C_j$ into two cells $c_1,c_2\in C'_j$ and denote by $\bowtie$ the balanced coloring in $L$ given by $c_1\bowtie c_2$. 
Since $[d]_{\bowtie}=d$ and $C_{j+1}=\{d\}$, $\bowtie$ is the unique balanced coloring such that $L/\bowtie=N$. 

Using Proposition~\ref{prop:equivtheta}, we construct a bifurcation branch $b\in \mathcal{B}(L,f)$ such that $b_{c_1}\neq b_{c_2}$. 
Let $A=\{i: \sigma_i^L(d)=c_2\}$, $\delta=\sign(f_{0\lambda}\sum_{i\in A} f_i)$ and
$$p_a=-1,\quad s_a=0,\quad\quad a\in C_0\cup \dots\cup C_{j-1} \cup C'_{j}\setminus\{c_2\},$$
$$p_{c_2}=0,\quad s_{c_2}=-\dfrac{2f_{0\lambda}}{f_{00}},\quad p_d=1,\quad s_d=-\sign\left(\delta \sum_{i=1}^k f_i\right)\dfrac{2}{f_{00}}\left|f_{0\lambda}\sum_{i\in A} f_i\right|^{2^{-1}},$$
$$p_a=l,\quad s_a=-\sign\left(\delta \sum_{i=1}^k f_i\right)\dfrac{2}{f_{00}}\left|\sum_{i=1}^k f_i\right|^{1-2^{-(l-1)}}\left|f_{0\lambda}\sum_{i\in A} f_i\right|^{2^{-l}},$$
for $a\in C_{j+l}$ and $2 \leq l\leq m-j$.
 We have that $(\delta,(p_c)_c, (s_c)_c)\in\Omega(L,f)$. By Proposition~\ref{prop:equivtheta}, there exists a bifurcation branch $b\in \mathcal{B}(L,f)$ such that $b_{c_1}\neq b_{c_2}$, since $p_{c_1}\neq p_{c_2}$. Thus $b\notin\Delta_{\bowtie}$ and $b$ is not lifted from $N$.
\end{proof}

The next example shows that the previous result is not always valid if the next layer to the one lifted has more than one cell. This example is very similar to Example~\ref{exe:uniqbalcolback}.% and we use an identical enumeration of the cells.

\begin{exe}\label{exe:liftinlastbutonelayerbackconenec}
Let $N$ be the feed-forward network in Figure~\ref{fig:liftinlayerbackconenecleft} and $L$ the feed-forward network in Figure~\ref{fig:liftinlayerbackconenecright}.  Consider the following balanced colorings in $L$: $\bowtie_1$ given by $2\bowtie_1 3$; $\bowtie_2$ given by $3\bowtie_2 4$; and $\bowtie_3$ given by $2\bowtie_3 4$. Then $N=L/\bowtie_1=L/\bowtie_2=L/\bowtie_3$. Note that $L$ is a lift inside the second layer. Let $f\in\mathcal{V}_0(N)$ generic and $b\in \mathcal{B}(L,f)$. Since $b_1$, $b_2$ and $b_3$ must have square-root-order $-1$ or $0$, we know that $b_1=b_2$ and $b\in \Delta_{\bowtie_1}$, $b_2=b_3$ and $b\in \Delta_{\bowtie_2}$ or $b_1=b_3$ and $b\in \Delta_{\bowtie_3}$. Therefore $b$ is lifted from $N$. Note that the third layer of $L$ has three cells.
\end{exe}

Next, we consider backward connected lifts inside a layer, except the first and the last two layers. 
The previous results already include the cases of a lift inside the first layer and a backward connected lift inside the last but one layer. 
And a lift inside the last layer breaks the backward connectedness.
In the next result, we see that there exists an open set of functions in $\mathcal{V}_0(N)$ such that there is a bifurcation branch on the lift which is not lifted from the original network, for lifts inside an intermediate layer.

\begin{prop}\label{prop:liftbifbrainliftinsidelayerusingbalcol}
Let $N$ be a feed-forward network with layers $C_0,\dots,C_m$, $f\in\mathcal{V}_0(N)$ generic and $L$ a feed-forward lift of $N$ such that $L$ is backward connected and a lift inside a layer $C_j$, where $0< j< m-1$.

If  $f_i>0$ for every $1\leq i\leq k$ (or $f_i<0$ for every $1\leq i\leq k$), then there is a bifurcation branch of $f$ on $L$ which is not lifted from $N$.
\end{prop}

\begin{proof}
Let $N$ be a feed-forward network with layers $C_0,\dots,C_m$, $f\in\mathcal{V}_0(N)$ generic and $L$ a feed-forward  lift of $N$ such that $L$ is backward connected and a lift inside a layer $C_j$, where $0< j< m-1$. 
Denote by $C'_j$ the $(j+1)$-layer of $L$ and by $(\sigma_i^L)_{i=1}^k$ the representative functions of $L$.
By Remark~\ref{obs:liftinslaycompsplit}, we assume that $L$ is the split of a cell $c\in C_0$ into two cells $c_1,c_2\in C'_0$ and denote by $\bowtie$ the balanced coloring in $L$ given by $c_1\bowtie c_2$. 
By Lemma~\ref{lem:unibalcolbackcon}, $\bowtie$ is the unique balanced coloring such that $L/\bowtie=N$. 

Assuming that $f_i>0$ for every $1\leq i\leq k$, we use Proposition~\ref{prop:equivtheta} to construct a bifurcation branch $b\in \mathcal{B}(L,f)$ such that $b\notin\Delta_{\bowtie}$. Define $\delta=\sign(f_{0\lambda})$, $p_a=-1$ and $s_a=0$, for $a\in C_0\cup\dots\cup C'_{j}\setminus \{c_1\}$, $p_{c_1}=0$ and $s_{c_1}=-2f_{0\lambda}/f_{00}$. We define the value of $p$ and $s$ by induction in the layers $C_{j+1},\dots,C_m$ in the following way: for $a\in C_l$, $j<l\leq m$, if $p_{\sigma^L_1(a)}=\dots=p_{\sigma^L_k(a)}=-1$ define $p_a=-1$ and $s_a=0$, otherwise define $p_a=\max\{p_{\sigma^L_1(a)}, \dots, p_{\sigma^L_k(a)}\}+1$ and
$$s_a=-\sign(f_{00}f_{0\lambda})\sqrt{\displaystyle -\frac{2\delta}{f_{00}}\sum_{i\in A(a)} f_i s_{\sigma_i(a)}},$$
where $A(a)=\{i: p_{\sigma^L_i(a)}= p_a-1\}$.

We have that $(\delta,(p_a)_a, (s_a)_a)\in\Omega(L,f)$ and $p_{c_1}\neq p_{c_2}$. By Proposition~\ref{prop:equivtheta}, there exists $b\in \mathcal{B}(L,f)$ such that $b\notin\Delta_{\bowtie}$. Thus there is a bifurcation branch of $f$ on $L$ not lifted from $N$.

The case $f_i<0$ for every $1\leq i\leq k$ is analogous. 
\end{proof}

Example~\ref{exe:liftinlastbutonelayerbackconenec} shows that the previous result is not valid if the lift is not backward connected. 
The network in Figure~\ref{fig:liftinlayerbackconenecleft} is not backward connected.

%
%\begin{exe}\label{exe:liftinlayerbackconenec}
%Let $N$ be the feed-forward network on Figure~\ref{fig:liftinlayerbackconenecright} and $L$ the feed-forward network on Figure~\ref{fig:liftinlayerbackconenecleft}.
%Consider the following balanced colorings in $L$: $\bowtie_1$ given by the class $\{1,2\}$, $\bowtie_2$ given by the class $\{2,3\}$ and $\bowtie_3$ given by the class $\{1,3\}$. 
%Then $N=L/\bowtie_1=L/\bowtie_2=L/\bowtie_3$. 
%Note that $L$ is a lift inside an intermediate layer.
%Let $f\in\mathcal{V}_0(N)$ generic and $b\in \mathcal{B}(L,f)$. Since $b_1$, $b_2$ and $b_3$ must have square-root-order $-1$ or $0$, we know that $b_1=b_2$ and $b\in \Delta_{\bowtie_1}$, $b_2=b_3$ and $b\in \Delta_{\bowtie_2}$ or $b_1=b_3$ and $b\in \Delta_{\bowtie_3}$. Therefore $b$ is lifted from $N$. Note that $L$ is not backward connected.
%\end{exe}

For lifts inside the second layer, we give sufficient conditions on the lift network structure and on the function $f\in\mathcal{V}_0(N)$ such that every bifurcation branch on the lift network is lifted from the original network.

\begin{prop}\label{prop:genkerincnonewbifseclay}
Let $N$ be a feed-forward network with layers $C_0,\dots,C_m$, $f\in\mathcal{V}_0(N)$ generic and $L$ a feed-forward lift of $N$ . Denote by $C'_1$ the second layer of $L$ and by  $(\sigma^L_i)_{i=1}^k$ the representative function of $L$. Assume that $L$ is the split of $c\in C_1$ into $c_1,c_2\in C'_1$ (and a lift inside $C_1$).

If for every $I\subseteq C'_1\setminus \{c_1,c_2\}$ there exist $d',d''\in C_2$ such that
$$(w_{I}^{d'}+w_1^{d'})(w_{I}^{d''}+w_1^{d''})<0 \wedge  (w_{I}^{d'}+w_2^{d'})(w_{I}^{d''}+w_2^{d''})<0,$$
where $w_{I}^{d}=\sum_{\sigma^L_i(d)\in I} f_i$, $w_{1}^{d}=\sum_{\sigma^L_i(d)=c_1} f_i$ and $w_{2}^{d}=\sum_{\sigma^L_i(d)=c_2} f_i$, then every bifurcation branch of $f$ on $L$ is lifted from $N$.
\end{prop}

\begin{proof}
Let $N$ be a feed-forward network with layers $C_0,\dots,C_m$, $f\in\mathcal{V}_0(N)$ generic and $L$ a feed-forward lift of $N$. Denote by $C'_1$ the second layer of $L$ and by  $(\sigma^L_i)_{i=1}^k$ the representative function of $L$. Assume that $L$ is the split of $c\in C_1$ into $c_1,c_2\in C'_1$.

We prove the result by contra position. Suppose that there exists $b\in \mathcal{B}(L,f)$ not lifted from $N$. Then $b_{c_1}\neq b_{c_2}$.
Let $(\delta,(p_a)_a,(s_a)_a)=\Theta(b)\in \Omega(L,f)$. For every $a\in C'_1$ we have that
$$p_a\in\{-1,0\}\quad \quad s_a=-(p_a+1)\dfrac{2 f_{0\lambda}}{f_{00}}.$$
Let $I=\{a\in C'_1\setminus \{c_1,c_2\}: p_a=0\}\subseteq C'_1\setminus \{c_1,c_2\}$. By \ref{cond:ome6}, for $d\in C_2$ such that $p_d=1$ we have that
$$s_d=\pm \dfrac{2}{f_{00}}\sqrt{\delta f_{0\lambda} \sum_{i\in A(d)} f_i},$$
where $A(d)=\{i: p_{\sigma^L_i(d)}=0\}$. Then $(\sum_{i\in A(d')} f_i)(\sum_{i\in A(d'')} f_i) >0$, if  $p_{d'}=p_{d''}=1$, $(\sum_{i\in A(d')} f_i)(\sum_{i\in A(d'')} f_i) =0$, if  $p_{d'}<1$ or $p_{d''}<1$, for every  $d',d''\in C_2$.
Thus $$\left(\sum_{i\in A(d')} f_i\right)\left(\sum_{i\in A(d'')} f_i\right) \geq 0,$$ for every $d',d''\in C_2$. Since $b_{c_1}\neq b_{c_2}$, $-1\leq p_{c_1}\neq p_{c_2}\leq 0$. If $p_{c_1}=0$ and $p_{c_2}=-1$, then $\sum_{i\in A(d)} f_i= w_I^d+w_1^d$. 
If $p_{c_1}=-1$ and $p_{c_2}=0$, then $\sum_{i\in A(d)} f_i= w_I^d+w_2^d$. So
$$(w_I^{d'}+w_1^{d'})(w_I^{d''}+w_1^{d''}) \geq 0 \vee (w_I^{d'}+w_2^{d'})(w_I^{d''}+w_2^{d''}) \geq 0,$$
for every $d',d''\in C_2$.
By contra position, we obtain the result.
\end{proof}

The next example shows that the previous condition is not necessary.

\begin{exe}\label{ex:sufcondisnotneccond2}
Returning to Example~\ref{ex:sufcondisnotneccond}, let $N$ be the feed-forward network of Figure~\ref{fig:sufcondisnotneccond}, $\bowtie$ the balanced coloring in $N$ given by the class $\{2, 3\}$ and $Q$ the quotient network of $N$ associated to $\bowtie$. 
%We use the same label for the cell in $N$ except the cell $[2]_{\bowtie}=[3]_{\bowtie}$ that we label as $2$.
The network $Q$ is a feed-forward network and $N$ is a lift inside the second layer.
Let $f\in\mathcal{V}_0(N)$ generic such that $(f_1+f_2)f_1>0$ and $f_3(f_2+f_3)<0$. 
Table~\ref{tab:sufcondisnotneccond} describes the possible bifurcation branches of $f$ on $N$. 
Examining the table, we can see that there is no bifurcation branch $b$ of $f$ on $N$ such that $b_2\neq b_3$. 
So $b\in\Delta_{\bowtie}$ and every bifurcation branch of $f$ on $N$ is lifted from $Q$.
However, the condition of Proposition~\ref{prop:genkerincnonewbifseclay} is not satisfied.
Let $w_2^4=f_1+f_2+f_3$, $w_2^5= f_2+f_3$, $w_2^6= f_3$, $w_3^4=0$, $w_3^5= f_1$ and $w_3^6= f_1+f_2$, then $w_3^4w_3^5=w_3^4w_3^5=0$, $w_3^6w_3^5=f_1(f_1+f_2)>0$.
\end{exe}

Next, we consider a lift inside an intermediate layer, except the second one, and give sufficient conditions on the lift network structure and on the function $f\in\mathcal{V}_0(N)$ such that no new bifurcation branch occurs besides the ones lifted from $N$. 
We will assume that the lift is a split of two cells which are the unique inputs cells of another two cells in the next layer.

\begin{prop}\label{prop:genkerincnonewbif}
Let $N$ be a feed-forward network with layers $C_0,\dots,C_m$, $f\in\mathcal{V}_0(N)$ generic, $L$ a feed-forward lift of $N$ and  $1<j\leq m-1$ . Denote by $C'_j$ the $(j+1)$ layer of $L$ and by $(\sigma^L_i)_{i=1}^k$ the representative function of $L$. Assume that $L$ is the split of $c\in C_j$ into $c_1,c_2\in C'_j$ (and a lift inside $C_j$).

If there exist $d',d''\in C_{j+1}$ such that $\sigma^L_i(d'),\sigma^L_i(d')\in\{c_1,c_2\}$, for every $1\leq i\leq k$, and 
$$w_1^{d'}w_1^{d''}<0\wedge w_2^{d'}w_2^{d''}<0\wedge w_1^{d'}w_1^{d''}+w_2^{d'}w_2^{d''}<w_1^{d'}w_2^{d''}+w_1^{d''}w_2^{d'},$$
where $w_{1}^{d}=\sum_{\sigma^L_i(d)=c_1} f_i$ and $w_{2}^{d}=\sum_{\sigma^L_i(d)=c_2} f_i$, then every bifurcation branch $b$ of $f$ on $L$ is lifted from $N$.
\end{prop}

\begin{proof}
Let $N$ be a feed-forward network with layers $C_0,\dots,C_m$, $f\in\mathcal{V}_0(N)$ generic, $L$ a feed-forward lift of $N$ and $1<j\leq m-1$. Denote by $C'_j$ the $(j+1)$ layer of $L$ and by $(\sigma^L_i)_{i=1}^k$ the representative function of $L$. Assume that $L$ is the split of $c\in C_j$ into $c_1,c_2\in C'_j$.
Let $b\in\mathcal{B}(L,f)$, $(\delta,(p_a)_a,(s_a)_a)=\Theta(b)\in \Omega(L,f)$ and $d',d''\in C_{j+1}$ such that $\sigma^L_i(d'),\sigma^L_i(d')\in\{c_1,c_2\}$, for $1\leq i\leq k$.
 
Suppose that $b_{c_1}\neq b_{c_2}$. Then $p_{c_1}=0\wedge p_{c_2}=-1$ or $p_{c_1}=-1\wedge p_{c_2}=0$ or $p_{c_1}=p_{c_2}>0\wedge s_{c_1}=-s_{c_2}$. We have that $w_1^{d'}w_1^{d''}\geq 0$,  if $p_{c_1}=0\wedge p_{c_2}=-1$. And $w_2^{d'}w_2^{d''}\geq 0$,  if $p_{c_1}=-1\wedge p_{c_2}=0$. If $p_{c_1}=p_{c_2}>0\wedge s_{c_1}=-s_{c_2}$, then $p_{d'}=p_{d'}=p_{c_1}+1$ and
$$s_{d'}=\pm \sqrt{-\dfrac{2\delta}{f_{00}} (w^{d'}_1-w^{d'}_2)s_{c_1}},\quad\quad s_{d''}=\pm \sqrt{-\dfrac{2\delta}{f_{00}} (w^{d''}_1-w^{d''}_2)s_{c_1}}.$$
Thus $(w^{d'}_1-w^{d'}_2)(w^{d''}_1-w^{d''}_2)>0$. Generically, $(w^{d'}_1-w^{d'}_2)(w^{d''}_1-w^{d''}_2)\neq 0$.

Therefore if $b\in\mathcal{B}(L,f)$ and $d',d''\in C_{j+1}$ such that $\sigma^L_i(d'),\sigma^L_i(d')\in\{c_1,c_2\}$, for $1\leq i\leq k$,
$w_1^{d'}w_1^{d''}<0\wedge w_2^{d'}w_2^{d''}<0\wedge w_1^{d'}w_1^{d''}+w_2^{d'}w_2^{d''}<w_1^{d'}w_2^{d''}+w_1^{d''}w_2^{d'}$, then $b_{c_1}= b_{c_2}$. And $b\in \mathcal{B}(L,f)$ is lifted from $N$.
\end{proof}

When the splitted cells only target one cell, the lift network structure can allow asynchronized bifurcation branches. 
When the splitted cells are not the unique source cells of two cells in the next layer, the lift network structure or the strength of the connections can allow asynchronized bifurcation branches. In the next example we present a lift network $L$ which is a split of two cells that are not the unique inputs cells of another two cells in the next layer and independently of the regular function there exists a bifurcation branch on $L$ which is not lifted from $N$.

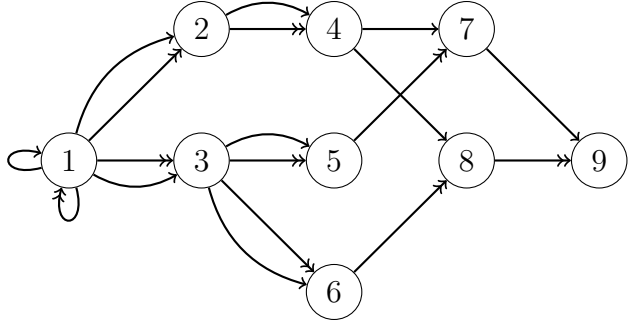
\begin{figure}[h]
\center
\begin{tikzpicture}%[node distance=1.6cm]
\node (n1) [circle,draw, label=center:1]   {\phantom{0}};
\node (n3) [circle,draw, label=center:3] [right=of n1]  {\phantom{0}};
\node (n2) [circle,draw, label=center:2] [above=of n3]  {\phantom{0}};
\node (n4) [circle,draw, label=center:4] [right=of n2]  {\phantom{0}};
\node (n51) [circle,draw, label=center:5] [right=of n3]  {\phantom{0}};
\node (n52) [circle,draw, label=center:6] [below=of n51] {\phantom{0}};
\node (n6) [circle,draw, label=center:7] [right=of n4]  {\phantom{0}};
\node (n7) [circle,draw, label=center:8] [right=of n51]  {\phantom{0}};
\node (n8) [circle,draw, label=center:9] [right=of n7]  {\phantom{0}};

\draw[->, thick] (n1) to [loop left] (n1);
\draw[->, thick] (n1) to  [bend left]  (n2);
\draw[->, thick] (n1) to  [bend right] (n3);
\draw[->, thick] (n2) to  [bend left] (n4);
\draw[->, thick] (n3) to  [bend left] (n51);
\draw[->, thick] (n3) to  [bend right](n52);
\draw[->, thick] (n4) to  (n6);
\draw[->, thick] (n4) to   (n7);
\draw[->, thick] (n6) to   (n8);

\draw[->>, thick] (n1) to [loop below] (n1);
\draw[->>, thick] (n1) to  (n2);
\draw[->>, thick] (n1) to  (n3);
\draw[->>, thick] (n2) to  (n4);
\draw[->>, thick] (n3) to  (n51);
\draw[->>, thick] (n3) to  (n52);
\draw[->>, thick] (n51) to  (n6);
\draw[->>, thick] (n52) to  (n7);
\draw[->>, thick] (n7) to  (n8);
\end{tikzpicture}
\caption{A network $L$ with a quotient network $N$ obtained by the balanced coloring $\bowtie$ given by $5\bowtie 6$. 
If $f\in\mathcal{V}_0(N)$, then there exists a bifurcation branch of $f$ on $L$ not lifted from $N$.
}
\label{fig:spiltnotuniqueinput}
\end{figure}

\begin{exe}\label{ex:spiltnotuniqueinput}
Let $L$ be the feed-forward network of Figure~\ref{fig:spiltnotuniqueinput}, $\bowtie$ the balanced coloring in $L$ given by $5 \bowtie 6$ and $N$ the quotient network of $L$ associated to $\bowtie$. 
%We use the same label for the cell in $N$ except the cell $[5.1]_{\bowtie}=[5.2]_{\bowtie}$ that we label as $2$.
The network $N$ is a feed-forward network and $L$ is a lift inside the third layer.
Let $f\in\mathcal{V}_0(N)$ generic. Then there exists a bifurcation branch of $f$ on $L$ not lifted from $N$.

Let $\delta=\sign(f_{0\lambda}(f_1+f_2))$, $p_1=p_3=p_{6}=-1$,  $p_2=p_{5}=0$, $p_4=1$, $p_7=p_8=2$, $p_9=3$, $s_1=s_3=s_{6}=0$, 
$$s_2=s_{5}=- \sign(f_1)\delta \dfrac{2|f_{0\lambda}|}{f_{00}},\quad s_4=-\sign(f_{0\lambda})\dfrac{2\sqrt{|f_{0\lambda}|}}{f_{00}}\sqrt{\left|f_1 +f_2 \right|},$$
$$s_7= s_8 = -\sign(f_{0\lambda})\dfrac{2\sqrt[4]{|f_{0\lambda}|}}{f_{00}}\sqrt{\displaystyle |f_1| \sqrt{\left|f_1 +f_2 \right|}}$$
and
$$s_9=\dfrac{2\sqrt[8]{|f_{0\lambda}|}}{f_{00}}\sqrt{\displaystyle |f_1 + f_2| \sqrt{\displaystyle |f_1| \sqrt{\left|f_1 +f_2 \right|}}}$$

%$$s_7= -\delta\sign(f_1-f_2)\sign(f_{00})\sqrt{\displaystyle -\frac{2\delta}{f_{00}} f_1 s_{4}}=-s_8$$
%and
%$$s_9= \sqrt{\displaystyle -\frac{2\delta}{f_{00}} (f_1 -f_2) s_{6}}.$$

Note that  $(\delta,(p_a)_a,(s_a)_a)\in\Omega(L,f)$. Let $b\in \mathcal{B}(L,f)$ be the bifurcation branch associated to $(\delta,(p_a)_a,(s_a)_a)$. Since $b$ has square-root order $3$, it can be lifted from $N$ if and only if $b_{5}=b_{6}$. However $p_{5}\neq p_{6}$ and $b$ is not lifted from $N$.
\end{exe}

\section*{Acknowledgments}
I thank Manuela Aguiar and Ana Paula Dias for helpful discussions and guidance throughout this work.
The research of the author was supported by FCT (Portugal) through the PhD grant PD/BD/105728/2014 and partially supported by CMUP (UID/ MAT/00144/ 2013), which is funded by FCT with national (MEC) and European structural funds (FEDER), under the partnership agreement PT2020.

\end{document}